\documentclass{article}
\usepackage[left=4cm, right=4cm, top=4cm, bottom=2.25cm]{geometry}
\usepackage{amsmath,color}
\usepackage{amsfonts,amsthm}
\usepackage{amssymb,enumerate,enumitem,verbatim}
\usepackage[pdftex]{graphicx}
\usepackage{hyperref}
\usepackage{amsmath,setspace,scalefnt}
\usepackage[usenames,dvipsnames,svgnames,table]{xcolor}
\usepackage{amsfonts,amsthm}
\usepackage{amssymb,enumitem,verbatim}
\usepackage{graphicx}
\usepackage{tikz}

\newcounter{capitalcounter}

\newcounter{claimcounter}

\newcounter{myfootnote}[page]

\newtheorem{lemma}{Lemma}[section]
\newtheorem{corollary}[lemma]{Corollary}
\newtheorem{theorem}[lemma]{Theorem}

\newtheorem{prop}[lemma]{Proposition}

\theoremstyle{definition}
\newtheorem{defn}[lemma]{Definition}

\newtheorem{claim}{Claim}

\theoremstyle{remark}
\newtheorem*{rmk}{Remark}

\newtheorem*{example}{Example}
\newtheorem*{nte}{Note}

\global\long\def\e{\varepsilon}
\global\long\def\eps{\varepsilon}

\global\long\def\N{\mathbb{N}}

\global\long\def\P{\mathbb{P}}

\global\long\def\E{\mathbb{E}}

\global\long\def\re{\begin{rmk}}
\global\long\def\mark{\end{rmk}}
\global\long\def\ex{\begin{example}}
\global\long\def\ple{\end{example}}
\global\long\def\no{\begin{nte}}
\global\long\def\ted{\end{nte}}
\global\long\def\en{\begin{compactenum}}
\global\long\def\um{\end{compactenum}}
\global\long\def\li{\begin{compactitem}}
\global\long\def\st{\end{compactitem}}
\global\long\def\de{\begin{defn}}
\global\long\def\fn{\end{defn}}
\global\long\def\cor{\begin{corollary}}
\global\long\def\ary{\end{corollary}}
\global\long\def\lem{\begin{lemma}}
\global\long\def\ma{\end{lemma}}
\global\long\def\arr{\begin{array}}
\global\long\def\ay{\end{array}}
\global\long\def\pr{\begin{proof}}
\global\long\def\oof{\end{proof}}

\global\long\def\lg#1{\log^{[{#1}]}n}

\newif\ifproofdone
\proofdonefalse
\newcounter{alabel}

\renewcommand\vec[1]{\overrightarrow{#1}}
\newcommand\cev[1]{\overleftarrow{#1}}

\title{Hamiltonicity in random directed graphs is born resilient}

\author{Richard Montgomery\footnote{University of Birmingham, Birmingham, B15 2TT, UK; r.h.montgomery@bham.ac.uk}}
\date{}

\begin{document}

\maketitle

\ifproofdone
\else
\begin{abstract}
Let $\{D_M\}_{M\geq 0}$ be the $n$-vertex random directed graph process, where $D_0$ is the empty directed graph on $n$ vertices, and subsequent directed graphs in the sequence are obtained by the addition of a new directed edge uniformly at random. For each $\eps>0$, we show that, almost surely, any directed graph $D_M$ with minimum in- and out-degree at least 1 is not only Hamiltonian (as shown by Frieze), but remains Hamiltonian when edges are removed, as long as at most $1/2-\eps$ of both the in- and out-edges incident to each vertex are removed. We say such a directed graph is \emph{$(1/2-\eps)$-resiliently Hamiltonian}. Furthermore, for each $\eps>0$, we show that, almost surely, each directed graph $D_M$ in the sequence is not $(1/2+\eps)$-resiliently Hamiltonian.

This improves a result of Ferber, Nenadov, Noever, Peter and \v{S}kori\'{c}, who showed, for each $\eps>0$, that the binomial random directed graph $D(n,p)$ is almost surely $(1/2-\eps)$-resiliently Hamiltonian if $p=\omega(\log^8n/n)$.
\end{abstract}
\fi


\section{Introduction}\label{intro}
\ifproofdone
\else

One of the most studied properties of graphs is that of \emph{Hamiltonicity}, the property that a graph contains a cycle through every vertex, known as a \emph{Hamilton cycle}. The natural extremal function for Hamiltonicity was studied by Dirac~\cite{dirac}, whose celebrated theorem demonstrates that any graph with $n\geq 3$ vertices and minimum degree at least $n/2$ is Hamiltonian. An early question of Erd\H{o}s and R\'enyi~\cite{ER59} in the study of the binomial random graph $G(n,p)$, where edges among $n$ vertices are chosen independently at random with probability $p$, asked when such a graph is likely to be Hamiltonian. After work by P\'osa~\cite{posa76} and by Korshunov~\cite{kor76}, this was determined independently by Koml\'os and Szemer\'edi~\cite{KSham} and Bollob\'as~\cite{bollo83}, who proved that if $p=(\log n+\log\log n+\omega(1))/n$, then $G(n,p)$ is Hamiltonian with probability $1-o(1)$. We say here that $G(n,p)$ is \emph{almost surely} Hamiltonian. This is best possible, for if $p=(\log n+\log\log n-\omega(1))/n$, then $G(n,p)$ almost surely has vertices of degree 0 or 1, and as such is clearly not Hamiltonian.

In fact, in $G(n,p)$, Hamiltonicity is almost surely concurrent with the property that every vertex has at least two neighbours. This is most precisely shown by the following beautiful result, proved independently by Bollob\'as~\cite{bollo84} and Ajtai, Koml\'os and Szemer\'edi~\cite{AKS}. Consider the $n$-vertex random graph process $G_0,\ldots,G_{\binom{n}2}$, where $G_0$ is an empty graph on $n$ vertices and each subsequent graph in the sequence is formed by the addition of a non-edge uniformly at random. Almost surely, the first graph in the sequence with minimum degree at least 2 is Hamiltonian~\cite{AKS,bollo84}.
Furthermore, we can strengthen this by showing that, in almost every random graph process, every graph with minimum degree at least 2 is not only Hamiltonian, but remains so despite the removal of any set of edges, subject only to a simple condition on the edges removed. That is, it is \emph{resiliently} Hamiltonian.

The general study of resilience in random graphs, initiated by Sudakov and Vu~\cite{SV08} in 2008, has developed into an active area of research (see, for example,~\cite{restree,benres,FK08,pancy,LS12,SV08} and the survey~\cite{sudsurvey}). We study the resilience of a graph $G$ with respect to some property $\mathcal{P}$ using the following definition.
\begin{defn} A graph $G$ is \emph{$\alpha$-resilient with respect to the property $\mathcal{P}$} if, for any $H\subset G$ with $d_H(v)\leq \alpha d_G(v)$ for each $v\in V(G)$, $G-H$ has property $\mathcal{P}$.
\end{defn}

Note that Dirac's theorem is exactly that the complete graph on $n\geq 3$ vertices is $(1/2)$-resiliently Hamiltonian. A natural generalisation to random graphs is to ask how resiliently Hamiltonian a typical random graph is. This was the subject of series of results (see \cite{benres,FK08,SV08}), before, in a key breakthrough, Lee and Sudakov~\cite{LS12} showed that, if $p=\omega(\log n/n)$, then $G(n,p)$ is almost surely $(1/2-o(1))$-resiliently Hamiltonian. Here, the constant $1/2$ is best possible as such a random graph can typically be disconnected while removing only $1/2+o(1)$ of the edges around any one vertex. However, the bound on $p$ can be improved slightly, and the result made best-possible by considering the resilience of Hamiltonicity in the random graph process.
Indeed, independently, the author~\cite{mehitres} and Nenadov, Steger and Truji\'c~\cite{nen17}, showed that in almost every $n$-vertex random graph process, each Hamiltonian graph is $(1/2-o(1))$-resiliently Hamiltonian. In this paper, we prove the corresponding result for the random \emph{directed} graph process.

A Hamilton cycle in a directed graph (digraph) is a cycle through every vertex whose edges are oriented in the same direction around the cycle. The corresponding result to Dirac's theorem was shown by Ghouila-Houri~\cite{GHdir}, who proved that every digraph on $n\geq 3$ edges with minimum in- and out-degree at least $n/2$ contains a Hamilton cycle. The binomial random digraph $D(n,p)$ has $n$ vertices and each possible edge chosen independently at random with probability $p$. For each vertex pair $u,v$, $\vec{uv}$ and $\vec{vu}$ may appear in $D(n,p)$.
 The techniques for studying Hamiltonicity in $G(n,p)$ do not immediately translate to the directed case, but an elegant general coupling argument of McDiarmid~\cite{McD83} shows that, if $p=(\log n+\log\log n+\omega(1))/n$, then $D(n,p)$ is almost surely Hamiltonian. However, the natural local impediment to Hamiltonicity in $D(n,p)$ is that every vertex must have in- and out-degree at least $1$.
This almost surely holds if $p=(\log n+\omega(1))/n$, and almost surely does not if $p=(\log n-\omega(1))/n$.

Frieze~\cite{frieze88} showed that, if $p=(\log n+\omega(1))/n$, then $D(n,p)$ is almost surely Hamiltonian, and gave a corresponding result for the random digraph process. In the $n$-vertex random digraph process $D_0,D_1,\ldots,D_{n(n-1)}$, $D_0$ is the empty digraph on $n$ vertices, and each subsequent digraph in the sequence is obtained by the addition of a new directed edge uniformly at random. Frieze~\cite{frieze88} showed that, in almost every $n$-vertex random digraph process, every digraph with minimum in- and out-degree at least $1$ is Hamiltonian.

To study resilience in directed graphs, we use the corresponding definition to resilience in graphs, as follows.

\begin{defn} A directed graph $D$ is \emph{$\alpha$-resilient with respect to the property $\mathcal{P}$} if, for any $H\subset D$ with $d^j_H(v)\leq \alpha d^j_D(v)$ for each $v\in V(D)$ and $j\in \{+,-\}$, $D-H$ has property $\mathcal{P}$.
\end{defn}

Hefetz, Steger and Sudakov~\cite{HSS16} showed that, if $p\gg \log n/\sqrt{n}$, then $D(n,p)$ is almost surely $(1/2-o(1)) $-resiliently Hamiltonian. As with the undirected case, the constant $1/2$ here is tight, but the bound on $p$ is rather loose. Ferber, Nenadov, Noever, Peter and \v{S}kori\'{c}~\cite{FNNPS17} showed that, if $p=\omega(\log^8n/n)$, then $D(n,p)$ is almost surely $(1/2-o(1))$-resiliently Hamiltonian. Here, we will make a best-possible improvement to the bound on $p$, and bring the known resilience of Hamiltonicity in random digraphs into line with that known for random graphs, as follows.

\begin{theorem}\label{directedresilience} Let $\e>0$. In almost every $n$-vertex random directed graph process $D_0,D_1,\ldots,D_{n(n-1)}$, the following is true for each $0\leq M\leq n(n-1)$. If $\delta^\pm(D_M)\geq 1$, then $D_M$ is $(1/2-\eps)$-resiliently Hamiltonian, but not $(1/2+\eps)$-resiliently Hamiltonian.
\end{theorem}

\noindent Standard techniques easily infer from Theorem~\ref{directedresilience} that, if $p=(\log n+\omega(1))/n$, then $D(n,p)$ is almost surely $(1/2-o(1))$-resiliently Hamiltonian (see, for example, Section~\ref{secrand}).

The constant $1/2$ in Theorem~\ref{directedresilience} arises from the following. Almost surely, if $D(n,p)$ has minimum in- and out-degree at least 1, then it can be disconnected into two roughly equal halves by deleting only a little over half of the in- and out-edges at each vertex. This is easy to show when $p=\omega(\log n/n)$, and, with a little care, it is possible to show in almost every random digraph process for each digraph $D_M$ with $\delta^\pm(D_M)\geq 1$ (see Section~\ref{secnonres}). Thus, it is relatively straightforward to demonstrate the limits of the resilience of Hamiltonicity required for Theorem~\ref{directedresilience}.

On the other hand, if we remove only at most a $(1/2-\eps)$-proportion of the in- and out-edges around each vertex, then we cannot disconnect the digraph $D$. In fact, we typically must retain two key properties. Firstly, if two large equal-sized vertex sets are chosen disjointly at random, then there is likely to be a matching directed from the first into the second. Secondly, given a small collection of pairs of vertices disjoint from a random small vertex subset, we can use the vertex subset to connect the pairs into a directed cycle. The first property allows us, by taking a sequence of random sets, to cover most of a typical random digraph by relatively few directed paths. The second property then allows us to join these paths together into a directed cycle, using a reserved random small set of vertices. This may, of course, not cover all the vertices, and hence we use the \emph{absorbing method}. This is described in detail in Section~\ref{secout}, but, in short, we note that the key behind our progress compared to Ferber, Nenadov, Noever, Peter and \v{S}kori\'{c}~\cite{FNNPS17}, who used the same broad outline, is in our construction of the \emph{reservoir}. In particular, each vertex in the reservoir is created by contracting a short directed path to create a new vertex. A Hamilton cycle in this altered digraph is found, before the contractions are undone to create a Hamilton cycle in the orginal digraph. This allows the use of a larger reservoir, and in combination with an adaptation of path connection methods by Glebov, Krivelevich and Johannsen~\cite{RomanPhD} to the directed graph setting (see Section~\ref{secpaths}), and the careful division of vertex sets using the local lemma (see Section~\ref{secloc}), makes the improvements required to show Theorem~\ref{directedresilience}.

In the rest of this section, we detail our notation. In Section~\ref{secout}, we give a sketch of our proof followed by an outline of the rest of the paper.

\subsection{Notation}

A digraph $D$ has vertex set $V(D)$ and edge set $E(D)$, and we set $|D|=|V(D)|$ and $e(D)=|E(D)|$.
For any set $A\subset V(D)$, we set $N^+_D(A)=\{v\in V(D)\setminus A:\exists u\in A\text{ s.t. }\vec{uv}\in E(D)\}$ and $N^-_D(A)=\{v\in V(D)\setminus A:\exists u\in A\text{ s.t. }\vec{vu}\in E(D)\}$. We say $N^+_D(A)$ is the out-neighbourhood of $A$ and $N_D^-(A)$ is the in-neighbourhood of $A$. Where $A$ is a single vertex $v$, we let $d^j_D(v)=|N^j_D(v)|$ for each $j\in \{+,-\}$. Given a set of edges $E$, we let $V(E)$ be the set of vertices contained in these edges. For any disjoint vertex sets $A$ and $B$ in a digraph $D$, and each $j\in \{+,-\}$, $e^+_D(A,B)$ is the number of edges directed from $A$ to $B$ in $D$, and $e_D^-(A,B)=e_D^+(B,A)$. Where it is clear from context, we often drop the digraph $D$ from the subscript. For a digraph $D$, $\Delta^-(D)$, $\Delta^+(D)$, $\delta^-(D)$ and $\delta^+(D)$ are the maximum in- and out-degree and the minimum in- and out-degree of $D$ respectively. For any vertex set $A$ in a digraph $D$, the digraph $D[A]$ has vertex set $A$ and edges exactly those in $D$ contained within $A$.

For convenience, we consider paths to have an inherent order, and thus treat them as an ordered sequence of vertices. In this sequence, we allow vertices to repeat consecutively without consequence. For example, if a path $P$ has start vertex $u$ and end vertex $v$, then we consider $uPv$ to be the same path as $P$. Given a path $P$, $\cev{P}$ is the path on the same vertices as $P$ but with the vertex order reversed. An \emph{alternating path} is one whose vertices all have either in-degree 0 or out-degree 0 within the path.
In a digraph $D$, for any disjoint vertex sets $A$ and $B$, a \emph{matching from $A$ into $B$} is a set of $|A|$ independent edges oriented from $A$ into $B$.

Where we use $\pm$ in an expression, we mean that this holds with $\pm$ replaced by both $+$ and $-$.
We use $\log$ for the natural logarithm and, for each $k\geq 2$, we use $\lg{k}$ to refer to the $k$th iterated logarithm of $n$, so that, for example, $\lg{3}=\log\log\log n$.  For each integer $k$, we let $[k]=\{1,\ldots,k\}$.

We use common asymptotic notation to relate functions of $n$, as follows. If $f=O(g)$ or $g=\Omega(f)$, then there exists a constant $C$ such that $f(n)\leq Cg(n)$ for every $n\in \N$. When the implicit constant $C$ depends on $\eps$, we will denote this in the subscript, using, for example $f=O_\eps(g)$. If $f=\omega(g)$ or $g=o(f)$, for a non-zero function $g$, then $f(n)/g(n)\to\infty$ as $n\to\infty$. When, for example, $\Omega(f)$ is used in expressions, we mean that this can be replaced by some function $g=\Omega(f)$ so that the expression holds. Many of our lemmas hold for $n\geq n_0(\eps)$, for some function $n_0$ depending on $\eps$. In the proofs we do not mention this explicitly, but only note that we take `$n$ sufficiently large' where our argument requires $n$ to be large. Similarly, when $f=\omega(g)$ or $f=o(g)$ we mean that this is true for each fixed $\eps$. If $f=O(g)$ and $g=O(f)$, then we say that $f=\Theta(g)$.

For clarity of presentation we do not include floor and ceiling symbols where they are not crucial.


\section{Outline and proof of a key lemma}\label{secout}

\subsection{Proof sketch and outline}\label{secoutsub}
\noindent\textbf{Pseudorandom digraphs.} We will build a Hamilton cycle in any sufficiently large digraph which satisfies certain pseudorandom properties, before showing that random digraphs resiliently contain such a digraph. These properties are defined precisely in Definition~\ref{pseuddefn}, but, roughly speaking, they say the following.

\begin{itemize}
\item The minimum and maximum in- and out-degrees are bounded (see \ref{pseud0}).
\item Small sets with many incident edges expand well  (see \ref{pseud1} and \ref{pseud2}).
\item Medium-sized sets expand to more than one half of the vertex set (see \ref{pseud3}).
\end{itemize}

\noindent
In our more informal discussion, we say a set expands if its in- and out-neighbourhood is comfortably larger than the set itself. The exact parameters of the expansion we use are found in Definition~\ref{pseuddefn}.

The first two properties listed above are naturally resilient (if the minimum degree bounds are reduced by an appropriate factor). The third condition is naturally almost surely $(1/2-o(1))$-resilient in $D(n,p)$ if $p=\omega(1/n)$. Typically, here, medium-sized sets will expand to almost all of the vertex set. Then, removing at most $1/2-\eps$ of the in- and out-degrees around any vertex will only reduce the size of the in- and out-neighbourhood by at most a factor of $1/2-\eps/2$, so the third condition holds resiliently.

\medskip

\noindent\textbf{Boosting the minimum degrees.} As we consider every Hamiltonian digraph in the random digraph process, we work with digraphs with very low minimum in- or out-degree. However, we use a natural modification to increase the minimum degree when there are a small number of vertices with low in- or out-degree. After the removal of edges, we take each low degree vertex and assign it both an in- and out-neighbour, before deleting the low degree vertex and merging its assigned in-neighbour into its assigned out-neighbour (see Definition~\ref{mergedefn}). This creates the pseudorandom digraph in which we find a Hamilton cycle. Taking this cycle, undoing the merging, and putting the low degree vertices between their assigned neighbours, creates a Hamilton cycle in the original digraph.

\medskip

\noindent\textbf{Hamilton cycles in pseudorandom digraphs.} We create a Hamilton cycle in a pseudorandom digraph $D$ using the same broad outline as Ferber, Nenadov, Noever, Peter and \v{S}kori\'{c}~\cite{FNNPS17}. We use the \emph{absorbing method}, first given as a general method by R\"odl, Ruci\'{n}ski and Szemer\'{e}di~\cite{RRS08}. We find a directed path $P$ in $D$ in combination with a reservoir $R$ in $V(D)\setminus V(P)$, so that, given any subset of vertices $R'\subset R$, we can find a directed path with vertex set $V(P)\cup R'$ and the same start and end vertices as $P$. Dividing the vertices $V(D)\setminus(V(P)\cup R)$ in the digraph into equal sized sets at random, we find matchings between them to create a small number of directed paths which cover the remaining vertices. Using vertices in $R$, we then join these paths into a directed cycle with $P$ -- say the cycle is $Q$. This gives a cycle covering all the vertices except for $R\setminus V(Q)$. Using the absorbing property we then find a path with vertex set $V(P)\cup (R\setminus V(Q))$ and the same end vertices as $P$, and replace $P$ with this path in $Q$ to get a Hamilton cycle.

The improvements we make from the methods of Ferber, Nenadov, Noever, Peter and \v{S}kori\'{c}~\cite{FNNPS17} come from three areas, as follows.
\begin{itemize}
\item  We use a more efficient absorbing structure so that the reservoir may have size $\Omega(n \lg{2}/\log n)$. Our reservoir in fact consists of disjoint directed paths, not vertices. We contract these paths into vertices in the obvious manner, and use these vertices as the reservoir. We then find a Hamilton cycle in the modified digraph, before replacing each contracted vertex by its original path to get a Hamilton cycle in the original digraph.

\item To construct absorbers and join paths into a cycle we develop and use a directed graph version of some path connection techniques by Glebov, Krivelevich and Johannsen~\cite{RomanPhD}.

\item We use the local lemma to randomly partition the vertex set into subsets and find matchings between them.
\end{itemize}

The first area represents the major innovation of this paper, while the subsequent two areas require quite a few technicalities. Due to this, we structure the paper so that the most important part of the argument appears first, in the rest of this section.

\medskip

\noindent\textbf{Outline.} In the rest of this section, we define our notion of pseudorandomness precisely, before defining a \emph{good partition}. We then prove a key lemma, Lemma~\ref{goodisgood}, that says any digraph with a good partition is Hamiltonian. This allows us to give the most important part of our argument, before embarking on the more technical aspects. Finally, we cover some useful results from the literature.

In Section~\ref{secloc}, we give our use of the local lemma to find useful vertex partitions. In Section~\ref{secpaths}, we give a digraph version of techniques by Glebov, Krivelevich and Johannsen~\cite{RomanPhD} for finding connecting paths. In Section~\ref{seccover}, we divide the vertex set into subsets and find matchings between them in order to cover most of the digraph with a small number of directed paths. In Section~\ref{secfindgood}, we combine this all to show that any sufficiently large pseudorandom digraph has a good partition.
Finally, then, in Section~\ref{secrand}, we find the pseudorandom properties in a random digraph needed to prove Theorem~\ref{directedresilience}, and also show the promised limits of resilience.

\subsection{Pseudorandom digraphs and good partitions}
We will begin by defining a \emph{$(d,\eps)$-pseudorandom} digraph and a \emph{good partition} of a digraph. A good partition is defined essentially as one with the properties needed to carry through our construction of a directed Hamilton cycle. On the other hand, the properties of a pseudorandom digraph more naturally reflect those of a typical random digraph. For example, for each $\eps>0$, if $p=\omega(\log n/n)$, then $D(n,p)$ is typically $(d,\eps)$-pseudorandom with $d=pn/2\log n$.

\begin{defn}\label{pseuddefn} An $n$-vertex digraph $D$ is \emph{$(d,\e)$-pseudorandom} if the following properties hold with  $m=n\log^{[3]} n/d\log n$.
\stepcounter{alabel}
\begin{enumerate}[label = \textbf{\Alph{alabel}\arabic{enumi}}]
\item\label{pseud0} $\delta^{\pm}(D)\geq d\log n$ and $\Delta^\pm(D)\leq 10^6d\log n$.
\item\label{pseud1} For each $j\in \{+,-\}$ and any disjoint sets $A,B\subset V(D)$, with $|A|\leq 2m$, and, for each $v\in A$, $d^j(v,B)\geq d\log^{[2]}n/\log^{[4]}n$, we have $|B|\geq 10|A|$.
\item\label{pseud2} For each  $j\in \{+,-\}$ and any disjoint sets $A,B\subset V(D)$, with $|A|\leq 2m$, and, for each $v\in A$, $d^j(v,B)\geq d(\log n)^{2/3}$, we have $|B|\geq (\log n)^{1/3}|A|$.
\item\label{pseud3} Every set $A\subset V(D)$ with $|A|=m$ satisfies $|N^{\pm}(A)|\geq (1/2+\e)n$.
\end{enumerate}
\end{defn}

We will show that every sufficiently large pseudorandom digraph is Hamiltonian, as follows.

\begin{theorem}\label{pseudorandom} For each $\e>0$, there exists some $n_0=n_0(\e)$ such that, for every $d\geq 10^{-5}$, every $(d,\e)$-pseudorandom digraph with at least $n_0$ vertices is Hamiltonian.
\end{theorem}

To show this, we will show that any sufficiently large pseudorandom digraph has a \emph{good partition} (see Lemma~\ref{goodpart}). This definition requires that directed cycles are found through particular edges. For this, we define the following weak and strong connection properties, where the key difference is that the latter property allows a cycle to be found through particular edges \emph{in a given order}.

\begin{defn}\label{weakdefn} A vertex set $U$ in a digraph $D$ is \emph{weakly connected} if, for any independent set $E$ of directed edges in the complete digraph with vertex set $U$, there is a directed cycle in $D+E$ which contains every edge in $E$.
\end{defn}

\begin{defn}\label{strongdefn} A vertex set $U$ in a digraph $D$ is \emph{strongly connected} if, for any $\ell$ and any independent set $E=\{e_1,\ldots,e_\ell\}$ of directed edges in the complete digraph with vertex set $U$, there is a directed cycle in $D+E$ which contains the edges $e_1,\ldots,e_\ell$ in that order.
\end{defn}

We also merge vertices using the following definition.

\begin{defn}\label{mergedefn} In a digraph $D$, we merge a vertex $x$ into a vertex $y$ by deleting $x$ and $y$ and creating a new vertex $z$ with in-neighbourhood $N_D^-(x)\setminus\{y\}$ and out-neighbourhood $N_D^+(y)\setminus\{x\}$.
\end{defn}

Using these definitions, we define a good partition as follows.

\begin{defn}\label{gooddefn}
In a digraph $D$, a vertex partition $V(D)=A\cup B_1\cup B_2\cup R_1\cup R_2\cup R_3\cup R_4$ is an \emph{$(\ell,r)$-good partition} if the following hold.
\stepcounter{alabel}
\begin{enumerate}[label = \textbf{\Alph{alabel}\arabic{enumi}}]
\item Any set $U\subset V(D)\setminus A$ with $|U|\leq 4r$ is strongly connected in $D[A\cup U]$. \label{D3}
\item If $B'\subset V(D)$ contains $B_1\cup B_2$, and $u,v\in B'\setminus (B_1\cup B_2)$ with $u\neq v$, then there is a collection of at most $\ell$ disjoint directed paths with length at least 1 in $D+\vec{uv}$ which cover $B'$ exactly, each start and end in $B_2$, and one of which contains the edge $\vec{uv}$. \label{D4}
\item There are matchings $M_1$, $M_2$ and $M_3$ from $R_2$ into $R_1$, $R_2$ into $R_3$ and $R_4$ into $R_3$ in $D$, respectively, and $|R_i|=r$ for each $i\in [4]$, so that the following holds. \label{D2}
\item Let $f:R_1\to R_4$ be such that, for each $v\in R_1$, $f(v)$ and $v$ are the end vertices of an alternating path in $M_1\cup M_2\cup M_3$. Merge each vertex $v\in R_1$ into $f(v)$ in $D$ to get the digraph $D'$. Let $R$ be the set of merged vertices in $D'$. Then, any set $U\subset B_2$ with $|U|\leq 2\ell$ is weakly connected in $D'[R\cup U]$. \label{D5}
\end{enumerate}
If a digraph has an $(\ell,r)$-good partition for some $\ell,r>0$, then we say $D$ \emph{has a good partition}.
\end{defn}

We now give the main part of our argument, showing that any digraph with a good partition is Hamiltonian.

\begin{lemma}\label{goodisgood}
Any digraph with a good partition is Hamiltonian.
\end{lemma}
\ifproofdone
\else
\begin{proof}
Let $D$ be a digraph and let $V(D)=A\cup B_1\cup B_2\cup R_1\cup R_2\cup R_3\cup R_4$ be an $(\ell,r)$-good partition of $D$, for some integers $\ell,r>0$. Using \ref{D2}, find a matching from $R_2$ into $R_1$, $R_2$ into $R_3$ and $R_4$ into $R_3$, so that \ref{D5} holds.
Use these matchings to label vertices so that $R_1=\{x_1,\ldots,x_r\}$,
 $R_2=\{u_1,\ldots,u_r\}$, $R_3=\{v_1,\ldots,v_r\}$ and $R_4=\{y_1,\ldots,y_r\}$, and, for each $i\in [r]$, $\vec{u_ix_i},\vec{u_iv_i},\vec{y_iv_i}\in E(D)$.

 By~\ref{D3} applied with $E=\{\vec{u_ix_i},\vec{y_iv_i}:i\in[r] \}$, we can find disjoint directed paths $P_i$ and $Q_i$, $i\in[r-1]$, and $P_r$, in $D[A\cup V(E)]$ so that
\begin{equation}\label{firstcycle}
u_1x_1P_1y_1v_1Q_1u_2x_2P_2y_2v_2Q_2\ldots u_r x_r P_r y_r v_r
\end{equation}
is a directed $u_1,v_r$-path in $D$ (see Figure~\ref{firstcyclepic}). Note that, for each $i\in [r]$, if $x_iP_iy_i$ is removed from the path in \eqref{firstcycle}, then, as $\vec{u_iv_i}\in E(D)$, this is still a directed $u_1,v_r$-path in $D$.

\begin{figure}[h]
\begin{center}
\begin{tikzpicture}[scale=1.1]

  \foreach \x in {1,3,5,7,9}
  {
  \draw ({\x},-0.55) -- ({\x-0.1},-0.45);
  \draw ({\x},-0.55) -- ({\x+0.1},-0.45);
  \draw ({\x+1},-0.45) -- ({\x+1-0.1},-0.55);
  \draw ({\x+1},-0.45) -- ({\x+1+0.1},-0.55);
  \draw [gray] ({\x},0) -- ({\x+1},0);
  \draw [gray] ({\x+0.55},0) -- ({\x+0.45},0.1);
  \draw [gray] ({\x+0.55},0) -- ({\x+0.45},-0.1);
  }

  \foreach \x in {1,...,10}
  {
  \draw [fill] (\x,0) circle [radius=0.1cm];
  }

  \foreach \x in {1,...,10}
  {
  \draw [fill] (\x,-1) circle [radius=0.1cm];
  }

  \foreach \x in {2,4,6,8}
  {
  \draw (\x,0) -- ({\x+1},0);
  \draw ({\x+0.5},0) -- ({\x+0.4},0.1);
  \draw ({\x+0.5},0) -- ({\x+0.4},-0.1);
  \draw ({\x+0.6},0) -- ({\x+0.5},0.1);
  \draw ({\x+0.6},0) -- ({\x+0.5},-0.1);
  }

  \foreach \x in {1,3,5,7,9}
  {
  \draw ({\x+0.5},-1) -- ({\x+0.4},-0.9);
  \draw ({\x+0.5},-1) -- ({\x+0.4},-1.1);
  \draw ({\x+0.6},-1) -- ({\x+0.5},-0.9);
  \draw ({\x+0.6},-1) -- ({\x+0.5},-1.1);
  }

  \foreach \x in {1,3,5,7,9}
  {
  \draw ({\x},-1) -- ({\x+1},-1);
  \draw ({\x},0) -- ({\x},-1);
  \draw ({\x+1},0) -- ({\x+1},-1);
  }

  \foreach \x in {1}
  {
  \draw ({2*\x-1},0.3) node {$u_{\x}$};
  \draw ({2*\x},0.3) node {$v_{\x}$};
  \draw ({2*\x-1},-1.3) node {$x_{\x}$};
  \draw ({2*\x},-1.3) node {$y_{\x}$};
  \draw ({2*\x-0.5},-1.3) node {$P_{\x}$};
  \draw ({2*\x+0.5},0.3) node {$Q_{\x}$};
  }

  \foreach \x in {5}
  {
  \draw ({2*\x-1},0.3) node {$u_{r}$};
  \draw ({2*\x},0.3) node {$v_{r}$};
  \draw ({2*\x-1},-1.3) node {$x_{r}$};
  \draw ({2*\x},-1.3) node {$y_{r}$};
  \draw ({2*\x-0.5},-1.3) node {$P_{r}$};
  }

  \foreach \x in {3}
  {
  \draw ({2*\x-1},0.3) node {$u_{i}$};
  \draw ({2*\x},0.3) node {$v_{i}$};
  \draw ({2*\x-1},-1.3) node {$x_{i}$};
  \draw ({2*\x},-1.3) node {$y_{i}$};
  \draw ({2*\x-0.5},-1.3) node {$P_{i}$};
  \draw ({2*\x+0.5},0.3) node {$Q_{i}$};
  }
\end{tikzpicture}
\end{center}
\vspace{-0.3cm}
\caption{The directed path in \eqref{firstcycle}, with the additional edges $\protect\vec{u_iv_i}$, $i\in [r]$, in grey.}\label{firstcyclepic}
\end{figure}

Let $R=R_1\cup R_2\cup R_3\cup R_4$ and note that every vertex in $R$ appears in \eqref{firstcycle}. Let $A'$ be the set of vertices in $A$ not appearing in \eqref{firstcycle}. By \ref{D4}, we can find some $m\in[\ell]$ and vertices and directed paths $s_iS_it_i$, $i\in[m]$, so that $\{s_iS_it_i:i\in [m]\}$ is a set of disjoint directed paths in $D+\vec{u_1v_r}$ with length at least $1$ which exactly covers $A'\cup B_1\cup B_2\cup\{u_1,v_r\}$ and for which $\{s_i,t_i:i\in [m]\}\subset B_2$,
 and so that $S_1$ contains the edge $\vec{u_1v_r}$. Say that $S_1=S_1'u_1v_r S_1''$. Thus, the following set of paths forms a partition of $A'\cup B_1\cup B_2\cup\{u_1,v_r\}$.
\begin{equation}\label{setofpaths}
\{s_1S_1'u_1\}\cup \{v_r S_1''t_1\}\cup\{s_iS_it_i:2\leq i\leq m\}
\end{equation}
Furthermore, then, the following paths form a partition of $V(D)$ (as depicted in Figure~\ref{firstcyclepic2}).
\begin{equation}\label{setofpaths2}
\{s_1S_1'u_1\}\cup \{v_r S_1''t_1\}\cup\{s_iS_it_i:2\leq i\leq m\}\cup \{x_iP_iy_i:i\in [r]\}\cup \{v_iQ_iu_{i+1}:1\leq i<r\}
\end{equation}

\begin{figure}[h]
\begin{center}
  \begin{tikzpicture}[scale=1.1]
  \foreach \x in {1,3,5,7,9}
{
\draw ({\x},-0.55) -- ({\x-0.1},-0.45);
\draw ({\x},-0.55) -- ({\x+0.1},-0.45);
\draw ({\x+1},-0.45) -- ({\x+1-0.1},-0.55);
\draw ({\x+1},-0.45) -- ({\x+1+0.1},-0.55);
\draw [gray] ({\x},0) -- ({\x+1},0);
\draw [gray] ({\x+0.55},0) -- ({\x+0.45},0.1);
\draw [gray] ({\x+0.55},0) -- ({\x+0.45},-0.1);
}

\foreach \x in {1,...,10}
{
\draw [fill] (\x,0) circle [radius=0.1cm];
}

\foreach \x in {1,...,10}
{
\draw [fill] (\x,-1) circle [radius=0.1cm];
}

\foreach \x in {2,4,6,8}
{
\draw (\x,0) -- ({\x+1},0);
\draw ({\x+0.5},0) -- ({\x+0.4},0.1);
\draw ({\x+0.5},0) -- ({\x+0.4},-0.1);
\draw ({\x+0.6},0) -- ({\x+0.5},0.1);
\draw ({\x+0.6},0) -- ({\x+0.5},-0.1);
}

\foreach \x in {1,3,5,7,9}
{
\draw ({\x+0.5},-1) -- ({\x+0.4},-0.9);
\draw ({\x+0.5},-1) -- ({\x+0.4},-1.1);
\draw ({\x+0.6},-1) -- ({\x+0.5},-0.9);
\draw ({\x+0.6},-1) -- ({\x+0.5},-1.1);
}

\foreach \x in {1,3,5,7,9}
{
\draw ({\x},-1) -- ({\x+1},-1);
\draw ({\x},0) -- ({\x},-1);
\draw ({\x+1},0) -- ({\x+1},-1);
}

\foreach \x in {1}
{
\draw ({2*\x-1},0.3) node {$u_{\x}$};
\draw ({2*\x},0.3) node {$v_{\x}$};
\draw ({2*\x-1},-1.3) node {$x_{\x}$};
\draw ({2*\x},-1.3) node {$y_{\x}$};
\draw ({2*\x-0.5},-1.3) node {$P_{\x}$};
\draw ({2*\x+0.5},0.3) node {$Q_{\x}$};
}

\foreach \x in {5}
{
\draw ({2*\x-1},0.3) node {$u_{r}$};
\draw ({2*\x},0.3) node {$v_{r}$};
\draw ({2*\x-1},-1.3) node {$x_{r}$};
\draw ({2*\x},-1.3) node {$y_{r}$};
\draw ({2*\x-0.5},-1.3) node {$P_{r}$};
}

\foreach \x in {3}
{
\draw ({2*\x-1},0.3) node {$u_{i}$};
\draw ({2*\x},0.3) node {$v_{i}$};
\draw ({2*\x-1},-1.3) node {$x_{i}$};
\draw ({2*\x},-1.3) node {$y_{i}$};
\draw ({2*\x-0.5},-1.3) node {$P_{i}$};
\draw ({2*\x+0.5},0.3) node {$Q_{i}$};
}
  \draw (0,0) -- (1,0);
  \draw (10,0) -- (11,0);

  \foreach \x in {0,11}
  {
  \draw [fill] (\x,0) circle [radius=0.1cm];
  }

  \foreach \x in {0,10}
  {
  \draw ({\x+0.5},0) -- ({\x+0.4},0.1);
  \draw ({\x+0.5},0) -- ({\x+0.4},-0.1);
  \draw ({\x+0.6},0) -- ({\x+0.5},0.1);
  \draw ({\x+0.6},0) -- ({\x+0.5},-0.1);
  }
  \

  \foreach \x in {5}
  {
  \draw ({2*\x+0.5},-0.4) node {$S_1''$};
  }

  \foreach \x in {1}
  {
  \draw ({2*\x-1.5},-0.4) node {$S_1'$};
  }

  \foreach \y in {-1,-2,-3}
  {
  \draw (0,{-\y}) -- (11,{-\y});
  }

  \foreach \y in {-1,-2,-3}
  {
  \foreach \x in {0,11}
  {
  \draw [fill] (\x,{-\y}) circle [radius=0.1cm];
  }
  }

  \foreach \y in {1,...,4}
  {
  \draw ({-0.3},{\y-1}) node {$s_\y$};
  \draw ({11+0.3},{\y-1}) node {$t_\y$};
  }

  \foreach \y in {-1,-2,-3}
  {
  \draw (5.6,{-\y}) -- (5.5,{-\y-0.1});
  \draw (5.6,{-\y}) -- (5.5,{-\y+0.1});
  \draw (5.5,{-\y}) -- (5.4,{-\y-0.1});
  \draw (5.5,{-\y}) -- (5.4,{-\y+0.1});
  }
  \foreach \y in {2,3,4}
  \draw (5.5,{\y-1.3}) node {$S_\y$};

  \draw (-0.5,3.5) -- (-0.5,-1.5) -- (11.5,-1.5) -- (11.5,3.5) -- (-0.5,3.5);

  \draw (-1.1,1) node {$V(D)$};

  \end{tikzpicture}
\end{center}
\vspace{-0.4cm}
\caption{The partition of $V(D)$ given by the paths in~\eqref{setofpaths2}. It remains to connect the $s_i,t_i$-paths in some order using some of the paths $x_jP_jy_j$, $j\in [r]$, to find a directed Hamilton cycle.}\label{firstcyclepic2}
\end{figure}
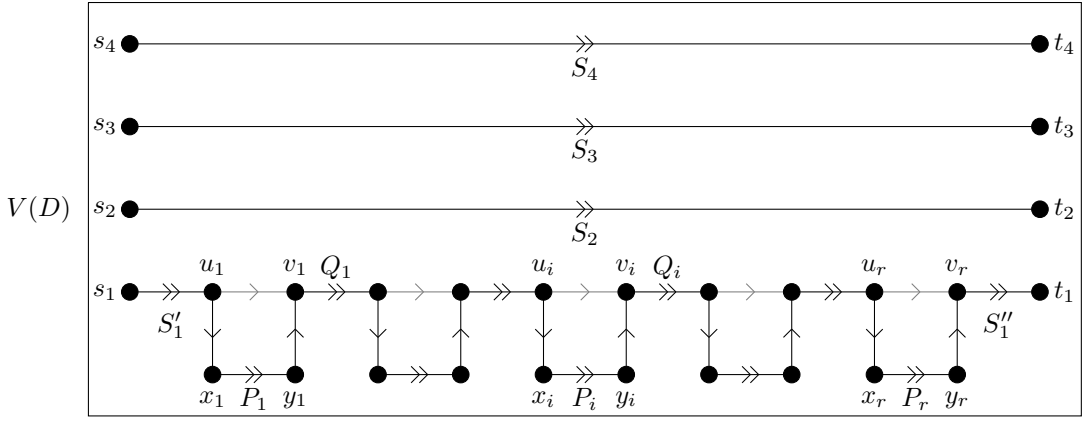

Let $D'$ be the digraph formed from $D$ by, for each $i\in [r]$, merging $x_i$ into $y_i$ to get the vertex $z_i$. Let $R'=\{z_i:i\in [r]\}$. Let $E'=\{\vec{s_it_i}:i\in[m]\}$. By \ref{D5}, there are disjoint directed paths $T_1,\ldots,T_{m}$ in $D'[R'\cup V(E')]$ and a bijection $g:[m]\to [m]$
with $g(1)=1$ so that
\[
t_{1}T_1s_{g(2)}t_{g(2)}T_2s_{g(3)}t_{g(3)}T_3\ldots s_{g(m)}t_{g(m)}T_ms_1
\]
is a directed path in $D'[R'\cup V(E')]+E'$.

For each $i\in [m]$, replace each vertex $z_j$, $j\in [r]$, of $T_i$ by the corresponding directed path $x_jP_jy_j$ and call the resulting path $T_i'$. Note that, from the definition of $D'$, $T_i'$ is a directed path in $D$. Thus, the path
\[
t_{1}T'_1s_{g(2)}t_{g(2)}T'_2s_{g(3)}t_{g(3)}T'_3\ldots s_{g(m)}t_{g(m)}T'_ms_1
\]
is a directed path in $D+E'$. Replacing each edge $s_it_i$, $2\leq i\leq m$ with the directed path $s_iS_it_i$, and adding the paths $v_r S_1''t_1$ and $s_1S_1'u_1$, we get that
\begin{equation*}\label{pathsetone}
C_1:=v_r S_1''t_{1}T'_1s_{g(2)}S_{g(2)}t_{g(2)}T'_2s_{g(3)}t_{g(3)}T'_3\ldots s_{g(m)}S_{g(m)}t_{g(m)}T'_m s_1S_1'u_1
\end{equation*}
is a directed path in $D$. As the paths in \eqref{setofpaths} form a partition of $A'\cup B_1\cup B_2\cup\{u_1,v_r\}$, $C_1$ has vertex set $A'\cup B_1\cup B_2\cup\{u_1,v_r\}$ with some sets $V(x_jP_jy_j)$ added (those appearing in some path $T'_i$).

For each $j\in [r]$, if $x_jP_jy_j$ is contained in $C_1$, then let $Z_j$ be the empty path on no vertices, and otherwise let $Z_j$ be $x_jP_jy_j$. Thus, the path
\begin{equation*}\label{pathsettwo}
C_2:=u_1Z_1v_1Q_1u_2Z_2v_2Q_2\ldots u_r Z_r  v_r
\end{equation*}
is the path in \eqref{firstcycle} with some paths $x_iP_iy_i$, $i\in [r]$, removed, which, by construction, is a directed path. The path $C_2$ contains exactly the vertices in the path in \eqref{firstcycle} except for those appearing in $C_1$, as well as $u_1$ and $v_r$. Thus, as the path $C_1$ contains all the vertices not in the path \eqref{firstcycle}, the two paths $C_1$ and $C_2$ form a cycle with vertex set $V(D)$, as required.
\end{proof}
\fi

\subsection{Preliminaries}
We will use the following well-known version of Chernoff's lemma (see, for example,~\cite[Corollary 2.3]{randomgraphs}).

\lem\label{chernoff} If $X$ is a binomial variable with standard parameters~$n$ and $p$, denoted $X=\mathrm{Bin}(n,p)$, and $\e$ satisfies $0<\e\leq 3/2$, then
\[
\P(|X-\E X|\geq \e \E X)\leq 2\exp\left(-\e^2\E X/3\right).\hfill\qedhere
\]
\ma

To find a matching from one set into another, we will use the following simple proposition (for undirected graphs).

\begin{prop}\label{matchymatchy} Let $G$ be a bipartite graph with vertex classes $A$ and $B$ with size $n$ each, such that, for each $U\subset A$ or $U\subset B$ with $|U|\leq \lceil n/2\rceil$, $|N(U)|\geq |U|$.
Then, there is a matching from $A$ to $B$ in $G$.
\end{prop}
\begin{proof}
Let $U\subset A$ with $\lceil n/2\rceil < |U|\leq n$. By considering a subset $U'\subset U$ with size $\lceil n/2\rceil$, we have that $|N(U)|\geq |N(U')|\geq \lceil n/2\rceil$. Thus, $|B\setminus N(U)|\leq n-\lceil n/2 \rceil=\lfloor n/2\rfloor$, so that, by the property in the lemma, $|N(B\setminus N(U))|\geq |B\setminus N(U)|$. Thus, as there are no edges between $U$ and $B\setminus N(U)$, we have
\[
|U|\leq n-|N(B\setminus N(U))|\leq n-|B\setminus N(U)|=n-(n-|N(U)|),
\]
and therefore $|N(U)|\geq |U|$.  By the property in the lemma, this is also true for all $U\subset A$ with $|U|\leq \lceil n/2\rceil$. Thus, Hall's matching condition is satisfied, and there is a matching from $A$ into $B$.
\end{proof}


\section{Set division using the local lemma}\label{secloc}
We will take vertex partitions using the following version of the Erd\H{os}-Lov\'asz local lemma, due to Lov\'asz~\cite[Theorem 1.1]{spencer}.

\begin{theorem}\label{genlocal} 
Let $A_1,\ldots,A_n$ be events in a probability space $\Omega$ with dependence graph $G$. Suppose there exist $0<x_1,\ldots,x_n<1$ such that, for each $i\in [n]$,
\[
\P(A_i)\leq x_i\prod_{j:ij\in E(G)}(1-x_j).
\]
Then, no such event $A_i$ occurs with strictly positive probability.
\end{theorem}

Through the following lemma, we use Theorem~\ref{genlocal} as follows. Given a vertex set $A$ in a digraph, where every vertex has plenty of in- and out-neighbours in $A$, we partition $A$ so that every vertex has at least some in- and out-neighbours in each subset in the partition. We use this in a similar manner to Hefetz, Krivelevich and Szab{\'o}~\cite{HKS12} on their work on the sharp threshold of certain spanning trees in $G(n,p)$.

\begin{lemma}\label{divisionlemma}
Let $\ell,m,n,\delta,\Delta\in \N$ with $2\leq \ell\leq \log n$, and let $\eps\in (0,1)$. Let $D$ be a digraph with $n$ vertices and $A\subset V(D)$ so that the following hold.
\stepcounter{alabel}
\begin{enumerate}[label = \textbf{\Alph{alabel}\arabic{enumi}}]
\item  For each $v\in V(D)$, $d^\pm(v,A)\geq \delta$ and $d^\pm(v)\leq \Delta$. \label{peep1}
\item For each $U\subset V(D)$ with $|U|=m$, $|N^{\pm}(U,A)|\geq (1/2+\e)|A|$. \label{peep2}
\end{enumerate}
Let $a=|A|$, and suppose that $a_i$, $i\in [\ell]$, are integers with $\sum_{i\in [\ell]}a_i\leq a$ such that the following hold for each $i\in [\ell]$.
\stepcounter{alabel}
\begin{enumerate}[label = \textbf{\Alph{alabel}\arabic{enumi}}]
\item\label{cond0} $\frac{\eps^2a_i^2}{a}\geq 10^3\ell^3$
\item\label{cond1} $\frac{\eps^2a_i^2}{a}\cdot\exp(\frac{a_i\delta}{24a})\geq 10^{5}\ell^3n$
\item\label{cond3} $\exp(\frac{ a_i\delta}{24a})\geq 320 \ell\Delta^2$
\item\label{cond2} $\frac{\eps^2 a_i}{10^3}\geq m\log (\frac{en}{m})$
\end{enumerate}

Then, there are disjoint sets $A_1,\ldots,A_\ell\subset A$ such that $|A_i|=a_i$ for each $i\in [\ell]$ and the following hold.
\stepcounter{alabel}
\begin{enumerate}[label = \textbf{\Alph{alabel}\arabic{enumi}}]
\item \label{finalprop1} For each $v\in V(D)$ and $i\in [\ell]$, $a_i\delta/4a\leq d^\pm(v,A)\leq 4a_i\Delta /a$.
\item \label{finalprop2} For each $U\subset V(D)$, with $|U|=m$, and $i\in [\ell]$, $|N^{\pm}(U,A_i)|\geq (1/2+\e/2)a_i$.
\end{enumerate}
\end{lemma}

\ifproofdone
\else
\begin{proof} Without loss of generality, assume that $a_1\geq a_2\geq \ldots \geq a_\ell$. Let $p=\eps a_\ell/10\ell a$. For each $i\in [\ell]$, let
\begin{equation}
p_i=(a_i/a)-p\geq (1-\eps/10)a_i/a.\label{nugget1}
\end{equation}
Let
\begin{equation}\label{nugget2}
p_0=\min\bigg\{1-\sum_{i\in [\ell]}p_i,(\ell+1)p\bigg\}\leq (\ell+1)p\leq \frac{\eps a_\ell}{5a}\overset{\eqref{nugget1}}{\leq} p_\ell.
\end{equation}
Noting that $\sum_{i=0}^\ell p_i\leq 1$, pick random disjoint sets
\[
B_0, B_1, B_2, \ldots, B_\ell\subset A
\]
so that, for each vertex $v\in V(D)$, $\P(v\in B_i)=p_i$ for each $0\leq i\leq \ell$, and whether $v$ appears in one of the sets, and which set it appears in, is independent of the location of all the other vertices in $D$.

We will show, using Theorem~\ref{genlocal}, that with positive probability the partition satisfies the following properties.

\stepcounter{alabel}
\begin{enumerate}[label = \textbf{\Alph{alabel}\arabic{enumi}}]
\item \label{sunny6} For each $i\in [\ell]$, $|B_i|\leq a_i$, and $|\cup_{i=0}^\ell B_i|\geq \sum_{i\in [\ell]}a_i$.
\item \label{sunny8} Every subset $U\subset V(D)$ with $|U|=m$ satisfies $|N^\pm(U,B_i)|\geq (1/2+\eps/2)a_i$ for each $i\in [\ell]$.
\item \label{sunny7}  For each $v\in V(D)$ and $i\in [\ell]$, $a_i\delta/4a\leq d^\pm(v,B_i)\leq 2a_i\Delta /a$.
\item \label{sunny9} For each $v\in V(D)$, $d^\pm(v,B_0)\leq 2a_\ell \Delta/a$.
\end{enumerate}

This will be sufficient to prove the lemma. Indeed, there will thus exist some partition in which  \ref{sunny6}--\ref{sunny9} hold. Let then $A_1,\ldots,A_\ell \subset A$ be disjoint subsets such that $B_i\subset A_i\subset B_i\cup B_0$ and $|A_i|=a_i$ hold for each $i\in [\ell]$; this is possible by~\ref{sunny6}.
As \ref{finalprop1} follows from \ref{sunny7} and \ref{sunny9}, and \ref{finalprop2} follows from \ref{sunny8}, we have the required partition.

Let then $B$ be the event that \ref{sunny6} or \ref{sunny8} does not hold. For each $v\in V(D)$, let $B(v)$ be the event that \ref{sunny7} or \ref{sunny9} does not hold for $v$. Let
\begin{equation}\label{clock}
q_B=1/2\;\;\;\text{ and }\;\;\;q=40\ell\exp(-a_\ell\delta/24a).
\end{equation}
Note that each event $B(v)$ has some dependence on $B$
and at most $4\Delta^2$ other events $B(v')$.

We will show the following two claims.
\begin{claim}\label{new2} For each $v\in V(D)$, $\P(B(v))\leq q(1-q_B)(1-q)^{4\Delta^2}$.
\end{claim}
\begin{claim}\label{new1} $\P(B)\leq q_B(1-q)^n$.
\end{claim}

Thus, by Theorem~\ref{genlocal} and Claims~\ref{new2} and~\ref{new1}, with positive probability some partition exists for which no event $B$ or $B(v)$, $v\in V(D)$, holds, and thus for which \ref{sunny6}--\ref{sunny9} hold, as required.
It remains then to prove the two claims.

\begin{proof}[Proof of Claim~\ref{new2}]
Let $v\in V(D)$, $j\in \{+,-\}$ and $d=d^j(v,A)$, so that, by \ref{peep1}, $\delta\leq d\leq \Delta$. For each $i\in [\ell]$, noting that, by \eqref{nugget1}, $a_i/2a\leq p_i\leq a_i/a$, we have, using Lemma~\ref{chernoff}, that
\begin{align}
\P(d^j(v,B_i)\notin (a_i\delta/4a,2a_i\Delta/a))
&\leq \P(|d^j(v,B_i)-p_id|>p_id/2) \nonumber
\\
&\leq 2\exp(-p_id/12)\leq 2\exp(-a_i\delta/24a)
\leq 2\exp(-a_{\ell}\delta/24a).\label{climate}
\end{align}
As, by \eqref{nugget2}, $p_0\leq p_\ell$, we also have that
\[
\P(d^j(v,B_0)>2a_{\ell}\Delta/a)\leq \P(d^j(v,B_\ell)>2a_\ell\Delta/a)\overset{\eqref{climate}}{\leq} 2\exp(-a_\ell\delta/24a).
\]

Thus, for each $v\in V(D)$, we have $\P(B(v))\leq 4(\ell+1)\exp(-a_{\ell}\delta/24a)\leq q/5$. By \ref{cond3}, we have that $4q\Delta^2\leq 1/2$. Thus,  as $q_B=1/2$,
\[
\P(B(v))\leq q\cdot (1-q_B)\cdot (1-4q\Delta^2)\leq q\cdot (1-q_B)\cdot (1-q)^{4\Delta^2}.\qedhere
\]
\end{proof}

We will prove Claim~\ref{new1} using two further claims.
\begin{claim}\label{new3} $\P(\ref{sunny6}\text{ holds})\leq 4\ell\exp(-\eps^2a_{\ell}^2/300\ell^2a)$.
\end{claim}
\begin{proof}[Proof of Claim~\ref{new3}]
By \eqref{nugget1} and Lemma~\ref{chernoff}, for each $i\in [\ell]$,
\begin{align}
\P(|B_i|> a_i)&= \P(|B_i|-p_ia>pa)
= \P(|B_i|-p_ia>(p/p_i)\cdot p_ia)
\leq 2\exp(-(p/p_i)^2\cdot p_ia/3)\nonumber
\\
&= 2\exp(-p^2a/3p_i)\leq 2\exp(-p^2a/3)= 2\exp(-\eps^2a_{\ell}^2/300\ell^2a).\label{tulipery}
\end{align}
Note that, for each $v\in A$, by \eqref{nugget1} and \eqref{nugget2}
\[
\P(v\in \cup_{i=0}^\ell B_i)=\sum_{i=0}^\ell p_i=\min\bigg\{1,\Big(\sum_{i\in[\ell]}\frac{a_i}{a}\Big)+p\bigg\}=:\bar{p}.
\]
If $\bar{p}=1$, then, with probability $1$, $|\cup_{i=0}^\ell B_i|=a\geq \sum_{i\in[\ell]}a_i$, as required, so assume that $\bar{p}=(\sum_{i\in[\ell]}a_i/a)+p$.
Then, by Lemma~\ref{chernoff},
\begin{align}
\P\bigg(|\cup_{i=0}^\ell B_i|< \sum_{i\in [\ell]}a_i\bigg)
&= \P\big(|\cup_{i=0}^\ell B_i|-\bar{p}a< -p a\big)\nonumber
= \P(|\cup_{i=0}^\ell B_i|-\bar{p}a< -p(\bar{p}a)/\bar{p})
\\
&
\leq 2\exp(-(p/\bar{p})^2\cdot (\bar{p}a)/3)
\leq 2\exp(-p^2a/3)= 2\exp(-\eps^2a_{\ell}^2/300\ell^2a).\label{tulipery2}
\end{align}
Thus, the claim follows from \eqref{tulipery2} and, for each $i\in [\ell]$, \eqref{tulipery}.
\end{proof}
\begin{claim}\label{new4}
For each $U\subset V(D)$ with $|U|=m$, $i\in [\ell]$, and $j\in \{+,-\}$,
\[
\P(|N^j(U,B_i)|< (1/2+\eps/2)a_i)\leq 2 \exp(-\eps^2a_i/400).
\]
\end{claim}
\begin{proof}[Proof of Claim~\ref{new4}] Let $U\subset V(D)$ with $|U|=m$, $i\in [\ell]$, and $j\in \{+,-\}$, and take $U'=N^j(U,A)$. By \ref{peep2}, $|U'|\geq (1/2+\eps)a$, so that
\[
\E|U'\cap B_i|\overset{\eqref{nugget1}}{\geq} (1/2+\eps)a\cdot (1-\eps/10)a_i/a\geq (1/2+3\eps/4)a_i\geq \frac{1/2+\eps/2}{1-\eps/8}a_i.
\]
Thus, we have both $\E|U'\cap B_i|\geq a_i/2$ and $(1/2+\eps/2)a_i\leq (1-\eps/8)\E|U'\cap B_i|$. Therefore, by Lemma~\ref{chernoff},
\[
\P(|U'\cap B_i|< (1/2+\eps/2)a_i)\leq 2\exp(-(\eps/8)^2\cdot \E|U'\cap B_i|/3)\leq 2\exp(-(\eps/8)^2\cdot a_i/6)\leq 2\exp(-\eps^2a_i/400).\qedhere
\]
\end{proof}
With these two claims, we can now prove Claim~\ref{new1}, as follows.

\medskip

\noindent\emph{Proof of Claim~\ref{new1}.}
By Claim~\ref{new3} and Claim~\ref{new4}, and taking into account that there are at most $\binom{n}{m}\leq (en/m)^m$ subsets with size $m$ of $V(D)$, we have
\begin{align*}
\P(B)&\leq 4\ell \exp\left(-\frac{\eps^2a_{\ell}^2}{300\ell^2a}\right)+4\ell\cdot \Big(\frac{en}{m}\Big)^m\cdot \exp\left(-\frac{\eps^2a_{\ell}}{400}\right)
\\
&\overset{\ref{cond2}}{\leq} 4\ell\exp\left(-\frac{\eps^2a_{\ell}^2}{300\ell^2a}\right)+4\ell \exp\left(-\frac{\eps^2a_{\ell}}{10^3}\right)\leq 8\ell\exp
\left(-\frac{\eps^2a_{\ell}^2}{300\ell^2a}\right).
\end{align*}
Note that, by \eqref{clock} and \ref{cond3}, $q\leq 1/2$. Therefore, as $q_B=1/2$, we have $q_B(1-q)^n\geq e^{-2qn}/2$. Thus, for Claim~\ref{new1}, it is sufficient to show that
\[
8\ell\exp\left(-\frac{\eps^2a_{\ell}^2}{300\ell^2a}\right)\leq e^{-2qn}/2,
\]
or, equivalently,
\[
\frac{\eps^2a_{\ell}^2}{300\ell^2a}-\log(16\ell)\geq 2qn.
\]
By \ref{cond0}, then, it is sufficient to show that
\[
\frac{\eps^2a_{\ell}^2}{300\ell^2a}\geq 4qn.
\]
However, as $q=40\ell\exp(-a_\ell\delta/24a)$, this holds directly from \ref{cond1}.
\hspace{4.7cm}{\textcolor{white}.}\qed
\end{proof}

\fi


\section{Path connection in pseudorandom digraphs}\label{secpaths}
In order to connect edges into a cycle, we develop a directed version of techniques of Glebov, Krivelevich and Johannsen~\cite{RomanPhD}. In~\cite{RomanPhD}, a concept of $(d,m)$-extendability is defined and used to find trees in any graph with certain expansion properties. In short, when a tree $S$ is $(d,m)$-extendable in a graph $G$ and $v\in V(S)$, then, subject to certain simple conditions, we can add a leaf to $v$ in $S$ so that the subsequent tree remains $(d,m)$-extendable. As shown in~\cite{RomanPhD}, this gives a flexible framework for embedding trees, but also allows paths to be found between vertex sets in $(d,m)$-extendable graphs. This latter property is what we want, except we will adapt this to work with directed graphs.

To do this, we work with two bipartite graphs, $H_1$ and $H_2$ say, with the same vertex classes, $A_1$ and $A_2$ say.
When we apply the results to a digraph $D$, $H_1$ will typically be the edges in $D$ directed from $A_1$ into $A_2$ (with the directions removed) while $H_2$ will be the edges in $D$ directed from $A_2$ into $A_1$. If the edges of a path alternate between $H_1$ and $H_2$ then this will be a directed path in the digraph.

To define our version of extendability we need the following definition.

\begin{defn}
Given a forest $S$, an edge $e\in E(S)$ and a vertex set $X\subset V(S)$ with exactly one vertex in each tree in $S$, let $d(e,X)$ be the distance of the shortest path from any vertex in $e$ to any vertex in $X$.
\end{defn}

Our extendable subgraph $S$ will be a forest lying in the union of the two bipartite graphs $H_1$ and $H_2$ mentioned above. We use the set $X$, with exactly one vertex per component tree of $S$, to record which edges of $S$ are in $H_1$ and which are in $H_2$. The edges will alternate between $H_1$ and $H_2$ working out from the vertex in $X$ in each component of $S$. The choice of $H_1$ and $H_2$ we later use, arising from a digraph $D$ as mentioned above, will imply that $S$ with the edges of each component tree directed out from the vertex in $X$ will lie in $D$.

We define extendability in a pair of bipartite graphs as follows.

\begin{defn}\label{extendability} Suppose that $H_1$ and $H_2$ are two bipartite graphs with the same vertex classes $A_1$ and $A_2$, and that $S$ is a forest containing $X\subset V(S)\cap A_1$ in which exactly one vertex in each tree in $S$ is in $X$. For each $i\in [2]$, let $S_i$ be the subgraph of $S$ with edge set $\{e\in E(S):d(e,X)\equiv i+1\mod 2\}$.

Then, we say $(S,X)$ is \emph{$(d,m)$-extendable in $(H_1,H_2)$} if the following hold.
\stepcounter{alabel}
\begin{enumerate}[label = \textbf{\Alph{alabel}\arabic{enumi}}]
\item $\Delta(S)\leq d$.\label{crack1}
\item For each $i\in [2]$, $S_i\subset H_i$.\label{crack2}
\item For each $i\in [2]$ and $U\subset A_i$ with $0<|U|\leq 2m$,\label{crack3}
\[
|N_{H_i}(U)\setminus V(S)|\geq d|U|-e_{S_i}(U,A_{3-i}).
\]
\item For each $i\in [2]$ and $U\subset A_i$ with $|U|\geq m$, $|N_{H_i}(U)|\geq |A_{3-i}|/2$.\label{crack4}
\end{enumerate}
\end{defn}

Given a $(d,m)$-extendable forest that is not too large, we can add an edge to any vertex with degree less than $d$ in the forest while remaining $(d,m)$-extendable, as follows.

\begin{lemma}\label{extendleaf} Let $d\geq 2$ and $m\geq 1$. Suppose that $H_1$ and $H_2$ are two bipartite graphs with the same vertex classes $A_1$ and $A_2$, and that $S$ is a forest containing $X\subset V(S)\cap A_1$ in which exactly one vertex in each tree in $S$ is in $X$. Suppose that $(S,X)$ is $(d,m)$-extendable in $(H_1,H_2)$, and
\begin{equation}\label{carbon}
|S|\leq \min\{|A_1|,|A_2|\}/2-2dm-2.
\end{equation}

 Then, for each $i\in [2]$ and $x\in V(S)\cap A_i$ with $d_S(x)<d$, there exists some $y\in N_{H_i}(x)\setminus V(S)$ so that $(S+xy,X)$ is $(d,m)$-extendable in $(H_1,H_2)$.
\end{lemma}
\ifproofdone
\else
\begin{proof} Suppose, to the contrary, that there is some $i\in [2]$ and $x\in V(S)\cap A_i$ with $d_S(x)<d$ for which no such $y$ exists. For each $y\in N_{H_i}(x)\setminus V(S)$, \ref{crack1}, \ref{crack2} and \ref{crack4} for $(S+xy,X)$ to be $(d,m)$-extendable in $(H_1,H_2)$ hold directly from the same statements for the extendability of $(S,X)$ and as $y\in N_{H_i}(x)\setminus V(S)$. Furthermore, as $V(S+xy)\cap A_i=V(S)\cap A_i$, \ref{crack3} for the index $3-i$ holds for $(S+xy,X)$.

Therefore, for each $y\in N_{H_i}(x)\setminus V(S)$, there is some set $U_y\subset A_i$ such that $|U_y|\leq 2m$ and
\begin{equation}\label{sorefoot}
|N_{H_i}(U_y)\setminus V(S+xy)|< d|U_y|-e_{S_i+xy}(U_y,A_{3-i}).
\end{equation}
From simple set relations and the extendability of $(S,X)$ in $(H_1,H_2)$, we have, for each $y\in N_{H_i}(x)\setminus V(S)$, that
\begin{align}
|N_{H_i}(U_y)\setminus V(S+xy)|+\mathbf{1}_{\{y\in N_{H_i}(U_y) \}}&=|N_{H_i}(U_y)\setminus V(S)|\nonumber\\
&\overset{\ref{crack3}}{\geq} d|U_y|-e_{S_i}(U_y,A_{3-i}) \nonumber\\
&=d|U_y|-e_{S_i+xy}(U_y,A_{3-i})+\mathbf{1}_{\{x\in U_y\}}.\label{newsore}
\end{align}
Thus, as \eqref{sorefoot} holds, we must have that $\mathbf{1}_{\{y\in N_{H_i}(U_y) \}}=1$, $\mathbf{1}_{\{x\in U_y\}}=0$, and that equality holds throughout \eqref{newsore}. That is, we have the following.
\begin{enumerate}[label = {\bfseries \Alph{alabel}5}]
\item For each $y\in N_{H_i}(x)\setminus V(S)$, we have $y\in N_{H_i}(U_y)$, $|N_{H_i}(U_y)\setminus V(S)|= d|U_y|-e_{S_i}(U_y,A_{3-i})$ and $x\notin U_y$.\label{crack5}
\end{enumerate}

We will now show that, in fact, each such $U_y$ must have size at most $m$, using the following claim.

\begin{claim}\label{arthur}
If $U\subset A_i$ and $m\leq |U|\leq 2m$, then $|N_{H_i}(U)\setminus V(S+xy)|>d|U|$.
\end{claim}
\begin{proof}[Proof of Claim~\ref{arthur}]
Let $U\subset A_i$ with $m\leq |U|\leq 2m$. Then, from \ref{crack4} and \eqref{carbon}, we have
\[
|N_{H_i}(U)\setminus V(S+xy)|\geq |A_{3-i}|/2-|S|-1\geq 2dm+1>d|U|.\qedhere
\]
\end{proof}
Thus, by \eqref{sorefoot} and Claim~\ref{arthur}, for each $y\in N_{H_i}(x)\setminus V(S)$, we have $|U_y|<m$.

\begin{claim}\label{claimthis}
For each $Y\subset N_{H_i}(x)\setminus V(S)$, $|\cup_{y\in Y}U_y|< m$ and
\[
|N_{H_i}(\cup_{y\in Y}U_y)\setminus V(S)|= d|\cup_{y\in Y}U_y|-e_{S_i}(\cup_{y\in Y}U_y,A_{3-i}).
\]
\end{claim}
\begin{proof}[Proof of Claim~\ref{claimthis}] We prove this by induction on $|Y|$. We know this to be true if $|Y|=1$ by \ref{crack5}, so assume that $|Y|>1$, and, picking $y_0\in Y$, that the claim is true for $Y':=Y\setminus \{y_0\}$ and $\{y_0\}$. Note first that $|\cup_{y\in Y}U_y|\leq |\cup_{y\in Y'}U_y|+|U_{y_0}|\leq 2m$.
Then, by the induction hypothesis and simple set relations (in particular, that, for all sets $A,B$ in a graph $G$, $|N(A\cup B)|+|N(A\cap B)|\leq |N(A)|+|N(B)|$), we have
\begin{align}
|N_{H_i}&(\cup_{y\in Y}U_y)\setminus V(S)|+|N_{H_i}((\cup_{y\in Y'}U_y)\cap U_{y_0})\setminus V(S)|
\nonumber\\
&\leq |N_{H_i}(\cup_{y\in Y'}U_y)\setminus V(S)|+|N_{H_i}(U_{y_0})\setminus V(S)|
\nonumber\\
&= d|\cup_{y\in Y'}U_y|-e_{S_i}(\cup_{y\in Y'}U_y,A_{3-i})+d|U_{y_0}|-e_{S_i}(U_{y_0},A_{3-i})\nonumber\\
&= d|\cup_{y\in Y}U_y|+d|(\cup_{y\in Y'}U_y)\cap U_{y_0}|-e_{S_i}(\cup_{y\in Y}U_y,A_{3-i})-e_{S_i}((\cup_{y\in Y'}U_y)\cap U_{y_0},A_{3-i}).\label{germane}
\end{align}
Now, by \ref{crack3} applied to $(\cup_{y\in Y'}U_y)\cap U_{y_0}$ (noting that this also holds if the set is empty), we have that
\[
|N_{H_i}((\cup_{y\in Y'}U_y)\cap U_{y_0})\setminus V(S)|\geq d|(\cup_{y\in Y'}U_y)\cap U_{y_0}|
-e_{S_i}((\cup_{y\in Y'}U_y)\cap U_{y_0},A_{3-i}).
\]
Therefore, in combination with \eqref{germane}, we have
\begin{equation}\label{resurge}
|N_{H_i}(\cup_{y\in Y}U_y)\setminus V(S)|
\leq d|\cup_{y\in Y}U_y|-e_{S_i}(\cup_{y\in Y}U_y,A_{3-i}).
\end{equation}
Thus, from \ref{crack3} applied to $\cup_{y\in Y}U_y$, we have that equality holds in \eqref{resurge}, as required. By Claim~\ref{arthur}, we also then have that $|\cup_{y\in Y}U_y|<m$.
\end{proof}

From Claim~\ref{claimthis} with $Y=N_{H_i}(x)\setminus V(S)$, we have $|\cup_{y\in Y}U_y|<m$ and
\begin{equation}\label{martha1}
|N_{H_i}(\cup_{y\in Y}U_y)\setminus V(S)|=d|\cup_{y\in Y}U_y|-e_{S_i}(\cup_{y\in Y}U_y,A_{3-i}).
\end{equation}
By \ref{crack5}, we have $y\in N_{H_i}(U_y)$ for each $y\in Y=N_{H_i}(x)\setminus V(S)$, and thus
\begin{equation}\label{martha2}
|N_{H_i}((\cup_{y\in Y}U_y)\cup\{x\})\setminus V(S)|=|N_{H_i}(\cup_{y\in Y}U_y)\setminus V(S)|.
\end{equation}
As, by  \ref{crack5}, $x\notin U_y$ for each $y\in Y$, and $d_{S}(x)<d$, we have
\begin{equation}\label{martha3}
d|(\cup_{y\in Y}U_y)\cup\{x\}|-e_{S_i}((\cup_{y\in Y}U_y)\cup\{x\},A_{3-i})
>
d|\cup_{y\in Y}U_y|-e_{S_i}(\cup_{y\in Y}U_y,A_{3-i}).
\end{equation}
Combining~\eqref{martha1}, \eqref{martha2} and \eqref{martha3}, we have
\[
d|(\cup_{y\in Y}U_y)\cup\{x\}|-e_{S_i}((\cup_{y\in Y}U_y)\cup\{x\},A_{3-i})>|N_{H_i}((\cup_{y\in Y}U_y)\cup\{x\})\setminus V(S)|.
\]
As $|(\cup_{y\in Y}U_y)\cup\{x\}|<m+1\leq 2m$, this contradicts \ref{crack3} in definition of the $(d,m)$-extendability of $(S,X)$ in $(H_1,H_2)$. Thus, some such $y$ as required by the lemma must exist.
\end{proof}
\fi

Applying Lemma~\ref{extendleaf} repeatedly, we can build an extendable copy of a tree, as follows.

\begin{lemma}\label{extendtree} Let $d\geq 2$ and $m\geq 1$. Let $A_1$ and $A_2$ be disjoint sets and let $X\subset A_1$ and $X'\subset X$. Let $T$ be a forest containing $X$, in which each component has exactly one vertex in $X$ and each vertex in $X'$ has degree 0. Let $T_x$, $x\in X'$, be vertex disjoint trees with maximum degree at most $d$ such that $T_x$ contains $x$. Let $T'=\cup_{x\in X'}T_x$.
Let $H_1$ and $H_2$ be bipartite graphs with vertex classes $A_1$ and $A_2$. Suppose $(T,X)$ is $(d,m)$-extendable in $(H_1,H_2)$ and that
\begin{equation}\label{import}
|T|+|T'|\leq \min\{|A_1|,|A_2|\}/2-2dm-1.
\end{equation}
Then, there is a copy $S$ of $T'$ so that each vertex $x\in X'$ is copied to itself, $V(T)\cap V(S)=X'$, and $(T+S,X)$ is $(d,m)$-extendable in $(H_1,H_2)$.
\end{lemma}
\ifproofdone
\else
\begin{proof} We will prove this by induction on $|E(T')|$. If $|E(T')|=0$, then let $S$ be the forest with vertex set $X'$ and no edges, noting this satisfies the lemma as $T+S=T$. Suppose then that $|E(T')|\geq 1$. Pick $x_0\in X'$ with $|E(T_{x_0})|\geq 1$, let $y_0$ be a leaf of $T_{x_0}$ which is not equal to $x_0$ and let $z_0$ be the neighbour of $y_0$ in $T_{x_0}$.
By the induction hypothesis, there is some copy $S'\subset H_1\cup H_2$ of $T'-y_0$ such that $x$ is copied to $x$, for each $x\in X'$, $V(T)\cap V(S')=X'$, and $(T+S',X)$ is $(d,m)$-extendable in $(H_1,H_2)$.

Let $z_0'$ be the copy of $z_0$ in $S'$. Suppose $z_0'\in A_1$, where the case where $z_0'\in A_2$ follows similarly. As $T_{x_0}$ has maximum degree at most $d$, the degree of $z_0$ in $S'$ is at most $d-1$. By \eqref{import},
\[
|T+S'|\leq |T|+|T'|-1\leq \min\{|A_1|,|A_2|\}/2-2dm-2.
\]
Thus, by Lemma~\ref{extendleaf}, there is some vertex $y_0'$ in $A_2\setminus V(T+S')$ so that $z_0'y_0'\in E(H_1)$ and, if $S:=S'+z_0'y_0'$, $(T+S,X)$ is $(d,m)$-extendable in $(H_1,H_2)$. Noting that $V(T)\cap V(S)=V(T)\cap V(S')=X'$, and $S$ is a copy of $T'$ in which each vertex in $X'$ is copied to itself, completes the proof.
\end{proof}
\fi

A critical part of the method used by Glebov, Krivelevich and Johannsen~\cite{RomanPhD} is that not only can we add a leaf and remain $(d,m)$-extendable, but we can also remove a leaf and remain $(d,m)$-extendable, as follows.

\begin{lemma}\label{removeleaf}
Suppose that $H_1$ and $H_2$ are two bipartite graphs with the same vertex classes $A_1$ and $A_2$, and that $S$ is a forest containing $X\subset V(S)\cap A_1$ in which exactly one vertex in each tree in $S$ is in $X$.
Suppose $j\in [2]$, $x\in V(S)\cap A_i$ and $y\in A_{3-j}\setminus V(S)$ so that $(S+xy,X)$ is $(d,m)$-extendable in $(H_1,H_2)$.

Then, $(S,X)$ is $(d,m)$-extendable in $(H_1,H_2)$.
\end{lemma}
\ifproofdone
\else
\begin{proof}
That \ref{crack1}, \ref{crack2} and \ref{crack4} hold for $(S,X)$ to be $(d,m)$-extendable in $(H_1,H_2)$ follows directly from the same conditions for the $(d,m)$-extendability of $(S+xy,X)$, so we need only check \ref{crack3}.

For each $i\in [2]$, let $S_j$ be the subgraph of $S$ with edge set $\{e\in E(S):d(e,X)\equiv i+1\mod 2\}$ and let $S_i'$ be the subgraph of $S+xy$ with edge set $\{e\in E(S+xy):d(e,X)\equiv i+1\mod 2\}$. Note that $S'_j=S_j+xy$ and $S'_{3-j}=S_{3-j}$.

Let $U\subset A_i$ with $0<|U|\leq 2m$. If $x\notin U$, then, as $(S+xy,X)$ is $(d,m)$-extendable in $(H_1,H_2)$,
\begin{align*}
|N_{H_j}(U)\setminus V(S)|&\geq |N_{H_j}(U)\setminus V(S+xy)|
\\
&\geq d|U|-e_{S_j+xy}(U,A_{3-j})
\\
&= d|U|-e_{S_j}(U,A_{3-j}).
\end{align*}
On the other hand, if $x\in U$, then $y\in N_{H_j}(U)$, so that
\begin{align*}
|N_{H_j}(U)\setminus V(S)|&= |N_{H_j}(U)\setminus V(S+xy)|+1
\\
&\geq d|U|-e_{S_j+xy}(U,A_{3-j})+1
\\
&= d|U|-e_{S_j}(U,A_{3-j}).
\end{align*}
Finally, let $U\subset A_{3-j}$ with $0<|U|\leq 2m$, then, as $y\in A_{3-j}$, and $S'_{3-j}=S_{3-j}$,
\begin{align*}
|N_{H_{3-j}}(U)\setminus V(S)| &=|N_{H_{3-j}}(U)\setminus V(S+xy)|\\
&\geq d|U|-e_{S'_{3-j}}(U,A_{3-j})\\
&=d|U|-e_{S_{3-j}}(U,A_{3-j}).
\end{align*}
Thus, \ref{crack3} holds so that $(S,X)$ is $(d,m)$-extendable in $(H_1,H_2)$.
\end{proof}
\fi

\subsection{Path connections}
We will now build paths between vertex sets, using the work in this section so far. The next lemma is the key lemma we use to show weak and strong connectivity in pseudorandom digraphs. We work with two extendable forests $S$ and $T$ in two vertex disjoint pairs of bipartite graphs $(G_1,G_2)$ and $(H_1,H_2)$. Given vertex sets $X'\subset V(S)$ and $Y'\subset V(T)$, we add a $(d-1)$-ary tree disjointly to each vertex in $X'$ and $Y'$ while retaining the respective extendability properties. The resulting trees attached to $X'$ and $Y'$ are, together, very large, and this will allow us to connect two trees from $X'$ and $Y'$ respectively using another set $Z$. This allows us to find a path from a vertex in $X'$ to a vertex in $Y'$ which initially alternates between edges in $G_1$ and $G_2$, then passes through $Z$, before alternating between edges in $H_1$ and $H_2$. The edges in these graphs will be chosen so that this will correspond to a directed path in our initial digraph.
Crucially, we can then use Lemma~\ref{removeleaf} to remove the vertices we added but did not use in this connection, while remaining extendable. This allows us to efficiently repeat the argument to find more connections.

We emphasise again that this is a fairly direct adaptation of the work of Glebov, Krivelevich and Johannsen~\cite{RomanPhD}, only in a form applicable to digraphs.

\begin{lemma}\label{findapath1} Let $m\geq 1$, $d\geq 3$ and $k$ satisfy $k=\lceil \log m/\log (d-1)\rceil$, and let  $0\leq j\leq k$. Let $G$ be a graph containing disjoint vertex sets
 $A_1$, $A_2$, $B_1$, $B_2$ and $Z$. Let $X\subset A_1$ and $Y\subset B_1$. Let $S$ and $T$ be forests so that there is exactly one vertex of $X$ and $Y$ in each component of $S$ and $T$ respectively, and so that
\begin{equation}\label{rest}
|S|,|T|\leq \frac{1}{2}\min\{|A_1|,|A_2|,|B_1|,|B_2|\}-10dm-1.
\end{equation}

Let $G_1$ and $G_2$ be bipartite subgraphs of $G$ with vertex classes $A_1$ and $A_2$ so that $(S,X)$ is $(d,m)$-extendable in $(G_1,G_2)$. Let  $H_1$ and $H_2$ be bipartite subgraphs of $G$ with vertex classes $B_1$ and $B_2$ so that $(T,Y)$ is $(d,m)$-extendable in $(H_1,H_2)$.

For any $U\subset A_1\cup A_2$ or $U\subset B_1\cup B_2$ with $|U|=m$, suppose that $|N_G(U,Z)|> |Z|/2$. Suppose that $X'\subset X$ and $Y'\subset Y$ with  $|X'|,|Y'|\geq m/(d-1)^j$ are sets of vertices with degree $0$ in $S$ and $T$ respectively.
Then, there are some $x\in X'$, $y\in Y'$, $x_0\in A_1\cup A_2$ and $y_0\in B_1\cup B_2$ and paths $P\subset G_1\cup G_2$ and $Q\subset H_1\cup H_2$ such that the following hold.
\begin{itemize}
\item $P$ is an $x,x_0$-path with length at most $j$ and no vertices in $V(S)\setminus\{x\}$, and $(S+P,X)$ is  $(d,m)$-extendable in $(G_1,G_2)$.
\item $Q$ is a $y,y_0$-path with length at most $j$ and no vertices in $V(T)\setminus \{y\}$, and $(T+Q,Y)$ is $(d,m)$-extendable in $(H_1,H_2)$.
\item $N_G(x_0)\cap N_G(y_0)\cap Z\neq \emptyset$.
\end{itemize}
\end{lemma}
\ifproofdone
\else
\begin{proof} By removing vertices from $X'$ and $Y'$ if necessary, assume that $|X'|=|Y'|=\lceil m/(d-1)^j\rceil$. Let $S_x$, $x\in X'$, be a collection of disjoint  $(d-1)$-ary trees with depth $j$ so that $S_x$ has root $x$ for each $x\in X'$, and let $S'=\cup_{x\in X'}S_x$. Let $T_y$, $y\in Y'$, be a collection of disjoint  $(d-1)$-ary trees with depth $j$ so that $T_y$ has root $y$ for each $y\in Y'$, and let $T'=\cup_{y\in Y'}T_y$.

Note that, by \eqref{rest}, $|S'|\leq 4md\leq \min\{|A_1|,|A_2|\}/2-2dm-1-|S|$.
Therefore, by Lemma~\ref{extendtree}, there is a copy $S''$ of $S'$ such that $x$ is copied to $x$ for each $x\in X'$ and $(S+S'',X)$ is $(d,m)$-extendable in $(G_1,G_2)$. Similarly, by Lemma~\ref{extendtree}, there is a copy $T''$ of $T'$ such that $y$ is copied to $y$ for each $y\in Y'$ and $(T+T'',Y)$ is $(d,m)$-extendable in $(H_1,H_2)$.

Noting that $|S''|,|T''|\geq m$, we can find some $x_0\in V(S'')$ and $y_0\in V(T'')$ such that $N_G(x_0)\cap N_G(y_0)\cap Z\neq\emptyset$. Let $x$ be the vertex in $X'$ for which there is a path, $P$ say, in $S''$ with length at most $j$ from $x$ to $x_0$. Note that, by iteratively removing leaves not in $X$, $S+S''$ can be turned into $S+P$. Thus, by repeated application of Lemma~\ref{removeleaf}, we have that $(S+P,X)$ is $(d,m)$-extendable in $(G_1,G_2)$.

Similarly, there is a $y,y_0$-path, $Q$ say,  in $T''$ with length at most $j$, so that $(T+Q,Y)$ is $(d,m)$-extendable in $(H_1,H_2)$. Thus, the vertices $x$, $y$, $x_0$, $y_0$ and paths $P$ and $Q$ satisfy the requirements in the lemma.
\end{proof}
\fi

\subsection{Strong connection in pseudorandom digraphs}
We now use Lemma~\ref{findapath1} to show a strong connection property in pseudorandom digraphs. We first use Lemma~\ref{divisionlemma} to find the sets to which we can then apply Lemma~\ref{findapath1}. In our application, we connect pairs of vertices with paths of length $O(\log n/\lg{2})$.

In the proof of the following theorem, and several times later in the paper, we make various calculations, often at length, to ensure that Lemma~\ref{divisionlemma} can be applied. For the reader who wishes to take these calculations as read, each such calculation begins with `We will now check the conditions', and we proceed after the calculation with a new paragraph continuing `Having checked the appropriate conditions'.

\begin{theorem}\label{shortpathpicky}  For each $\e>0$, there exists $n_0=n_0(\e)$ such that the following holds for every $d\geq 10^{-5}$ and $n\geq n_0$. Let $D$ be an $n$-vertex $(d,\eps)$-pseudorandom digraph containing a set $A$ with $|A|\geq n/\lg{3}$ and in which the following hold with $m=n\log^{[3]} n/d\log n$.
\begin{enumerate}
\item For each $v\in V(D)$, $d^\pm(v,A)\geq d(\log n)^{3/4}$.
\item For each $U\subset V(D)$ with $|U|=m$, $|N^\pm(U,A)|\geq (1/2+\eps/2)|A|$.
\end{enumerate}
Then, any set $V\subset V(D)\setminus A$ with
\[
|V|\leq \frac{|A|\lg{2}}{\log n\cdot \lg{7}}
\]
is strongly connected in $D[A\cup V]$.
\end{theorem}
\ifproofdone
\else
\begin{proof} Let $\alpha=\lfloor |A|/5\rfloor$, $\delta=d(\log n)^{3/4}$ and $\Delta=10^6d\log n$ so that, for each $v\in V(D)$, $d^\pm(v,A)\geq \delta$, and, by \ref{pseud0} in the definition of $(d,\eps)$-pseudorandomness, we have $\Delta^{\pm}(D)\leq\Delta$.

We will now check the conditions \ref{cond0}--\ref{cond2} to apply Lemma~\ref{divisionlemma} with $\delta$, $\Delta$, $m$, $n$ unchanged, $\ell=5$, $a_1=\ldots=a_5=\alpha$ and $\eps/2$ in place of $\eps$.
We have, for \ref{cond0}--\ref{cond2} in turn, that
\[
\frac{\eps^2\alpha^2}{4|A|}=\Omega_\eps(\alpha)=\omega(1),
\]
\[
\frac{\eps^2\alpha^2}{4|A|}\cdot \exp \left(\frac{\alpha\delta}{24|A|}\right)\geq \frac{\eps^2\alpha^2}{4|A|}\cdot\exp \left(\frac{\delta}{200}\right)=
\Omega_\eps\left(\frac{n}{\lg{3}}\cdot\exp\left(\frac{d (\log n)^{3/4}}{200}\right)\right)=\omega(n),
\]
\[
\exp \left(\frac{\alpha\delta}{24|A|}\right)\geq \exp \left(\frac{\delta}{200}\right)=\exp(\Omega(d(\log n)^{3/4}))=\omega(d^2\log^2 n)=\omega(\Delta^2),
\]
and, as $\log (en/m)=O(\log (d\log n))$,
\[
\frac{\eps^2\alpha}{4\cdot 10^3}=\Omega_\eps(\alpha)=\Omega_\eps\left(\frac{n}{\lg{3}}\right)
=\Omega_\eps\left(m\cdot\frac{d\log n}{(\lg{3})^2}\right)=\omega(m\cdot\log(d\log n))=\omega\left(m\cdot \log\left(\frac{en}{m}\right)\right).
\]

Having checked the appropriate conditions, by Lemma~\ref{divisionlemma}, for sufficiently large $n$, there are disjoint vertex sets $A'_1,A_2,B'_1,B_2,Z$ in $A$, each with size $\alpha$, such that
\stepcounter{alabel}
\begin{enumerate}[label = \textbf{\Alph{alabel}\arabic{enumi}}]
\item For each $v\in V(D)$ and $B\in \{A'_1,A_2,B_1',B_2,Z\}$, $d^\pm(v,B)\geq \alpha\delta/4|A|\geq d(\log n)^{3/4}/40$.\label{sigh1}
\item For each $U\subset V(D)$ with $|U|=m$  and $B\in \{A'_1,A_2,B_1',B_2,Z\}$, $|N^\pm(U,B)|\geq (1/2+\eps/4)|B|$.\label{sigh2}
\end{enumerate}

Now, let
\begin{equation}\label{l0defn}
\ell_0:=\frac{|A|\lg{2}}{\log n\cdot \lg{7}}.
\end{equation}
and take an arbitrary set $E=\{\vec{y_ix_i}:i\in [\ell]\}$ of $\ell\leq \ell_0$ vertex disjoint edges in the complete digraph with vertex set $V(D)\setminus A$. To show the lemma, it is sufficient to find a directed cycle in $D[A\cup V(E)]+E$ which contains the edges $\vec{y_1x_1},\ldots,\vec{y_\ell x_\ell}$ in order.

Let $X=\{x_1,\ldots,x_\ell\}$ and $Y=\{y_1,\ldots y_\ell\}$. Let $G_1$ and $G_2$ be the bipartite (undirected) graphs with vertex set $A_1:=A'_1\cup X$ and $A_2$ and the edges in $D$ from $A_1$ to $A_2$ and $A_2$ to $A_1$ respectively, but without their directions. Let $H_1$ and $H_2$ be the bipartite (undirected) graphs with vertex set $B_1:=B'_1\cup Y$ and $B_2$ and the edges in $D$ from $B_2$ to $B_1$ and $B_1$ to $B_2$ respectively, again without their directions.
Let $G$ have vertex set $A\cup X\cup Y$ and consist of the edges in $G_1, G_2, H_1, H_2$ as well as the edges from $A_1\cup A_2$ to $Z$ and $Z$ to $B_1\cup B_2$. Let $d_0=(\log n)^{1/3}$ and note that $d_0m=o(a)$. Let $I_X$ and $I_Y$ respectively be the graphs with vertex set $X$ and $Y$ which each have no edges.

\begin{claim}\label{extendclaim} For sufficiently large $n$,
$(I_X,X)$ is $(d_0,m)$-extendable in $(G_1,G_2)$ and $(I_Y,Y)$ is $(d_0,m)$-extendable in $(H_1,H_2)$.
\end{claim}
\begin{proof}[Proof of Claim~\ref{extendclaim}] We will show that $(I_X,X)$ is $(d_0,m)$-extendable in $(G_1,G_2)$. That $(I_Y,Y)$ is $(d_0,m)$-extendable in $(H_1,H_2)$ follows similarly. Note that \ref{crack1} and \ref{crack2} in Definition~\ref{extendability} are immediate as $I_X$ and $I_Y$ have no edges. Furthermore, from \ref{sigh2}, for each $U\subset A_2$ with $|U|\geq m$, we have
\[
|N_{G_2}(U)\setminus X|=|N_{D}^+(U,A_1\setminus X)|=|N_{D}^+(U,A_1')| \geq (1/2+\eps/4)|A_1'|= (1/2+\eps/4)(|A_1|-|X|)> |A_1|/2,
\]
where we have used that $|X|=\ell =o(a)$.
Thus, as for each $U\subset A_1$ with $|U|\geq m$ we have, from \ref{sigh2}, that $|N_{G_1}(U)|=|N_D^+(U,A_2)|>|A_2|/2$, we have that \ref{crack4} holds.

Let $U\subset A_1$ with $0<|U|\leq 2m$. Each vertex in $U$ has at least $d(\log n)^{3/4}/40\geq d(\log n)^{2/3}$ out-neighbours in $A_2$ in $D$ by \ref{sigh1}, for sufficiently large $n$. Thus, as $A_2\cap X=\emptyset$, by \ref{pseud2} in the definition of $(d,\eps)$-pseudorandomness, we have
\begin{equation}\label{this1}
|N_{G_1}(U,A_2\setminus X)|=|N_{D}^+(U,A_2)|\geq d_0|U|.
\end{equation}
Furthermore, for each $U\subset A_2$ with $0<|U|\leq 2m$, each vertex in $U$ has at least $d(\log n)^{3/4}/40\geq d(\log n)^{2/3}$ out-neighbours in $A_1'$ in $D$ by \ref{sigh1}, for sufficiently large $n$. Thus, by \ref{pseud2} in the definition of $(d,\eps)$-pseudorandomness, we have
\begin{equation}\label{that1}
|N_{G_2}(U,A_1\setminus X)|=|N^+_{D}(U,A_1')|\geq d_0|U|.
\end{equation}
In combination, \eqref{this1} and \eqref{that1} show that \ref{crack3} holds to complete the proof of the $(d_0,m)$-extendability of $(I_X,X)$ in $(G_1,G_2)$.
\end{proof}

Let
\begin{equation}\label{meaculpa}
k=\left\lceil \frac{\log m}{\log (d_0-1)}\right\rceil=O\left(\frac{\log n}{\lg{2}}\right)\overset{\eqref{l0defn}}{=}o\left(\frac{a}{\ell_0}\right).
\end{equation}
Take a maximal set $I\subset [\ell]$ for which there is a vertex disjoint set of paths $P_i\subset G_1\cup G_2$ and $Q_i\subset H_1\cup H_2$, $i\in I$, with length at most $k$ each and distinct vertices $z_i$, $i\in I$, in $Z$ such that the following hold (with addition modulo $\ell$ in the indices).
\stepcounter{alabel}
\begin{enumerate}[label = \textbf{\Alph{alabel}\arabic{enumi}}]
\item For each $i\in I$, $P_i$ is a path with start vertex $x_i$ and no vertices in $X\setminus\{x_i\}$, and $(I_X+\sum_{i\in I}P_i,X)$ is $(d_0,m)$-extendable in $(G_1,G_2)$.\label{other1}
\item For each $i\in I$, $Q_i$ is a path with start vertex $y_{i+1}$ and no vertices in $Y\setminus\{y_{i+1}\}$, and $(I_Y+\sum_{i\in I}Q_i,Y)$ is $(d_0,m)$-extendable in $(H_1,H_2)$.\label{other2}
\item For each $i\in I$, $P_iz_i\cev{Q_i}$ is a directed $x_i,y_{i+1}$-path in $D$.\label{other3}
\end{enumerate}
Note that, by Claim~\ref{extendclaim}, $I=\emptyset$ satisfies \ref{other1}--\ref{other3}, so that such a set $I$ exists.

If $I=[\ell]$, then, by \ref{other3}, $\sum_{i\in I}(y_ix_iP_iz_i\cev{Q}_i)$ is a directed cycle in $E+D[A\cup V(E)]$ containing the edges $y_ix_i$, $i\in I$, in order, as required. Thus, assume for contradiction that there is some $j\in [\ell]\setminus I$.

Let the paths $P_i$ and $Q_i$, and disjoint vertices $z_i$, $i\in I$, satisfy the conditions above. Let $Z'=Z\setminus (\cup_{i\in I}z_i)$, and note that $|Z'|\geq |Z|-\ell$. Thus, by \ref{sigh2}, the definition of $G$, and as $\ell\leq \ell_0=o(a)$, we have the following.
\begin{enumerate}[label = {\bfseries \Alph{alabel}4}]
\item For any $U\subset A_1\cup A_2$ or $U\subset B_1\cup B_2$ with $|U|=m$, we have $|N_G(U,Z')|\geq (1/2+\eps/4)|Z|-\ell>|Z|/2\geq |Z'|/2$.\label{other4}
\end{enumerate}
Let $S=I_X+\sum_{i\in I}P_i$ and $T=I_Y+\sum_{i\in I}Q_i$, so that, by \ref{other1} and \ref{other2}, $(S,X)$ and $(T,Y)$ are $(d_0,m)$-extendable in $(G_1,G_2)$ and $(H_1,H_2)$ respectively. Note that, by \eqref{meaculpa}, as $d_0m=o(a)$, we have
\begin{equation}\label{weird}
|S|,|T|\leq \ell(k+1)=o(a)=o\left(\frac{1}{2}\min\{|A_1|,|A_2|,|B_1|,|B_2|\}-10d_0m-1\right).
\end{equation}
Thus, by \eqref{meaculpa}, \eqref{weird} and \ref{other4}, we can apply Lemma~\ref{findapath1} with $X'=\{x_j\}$ and $Y'=\{y_{j+1}\}$ to get vertices $x'\in A_1\cup A_2$, $y'\in B_1\cup B_2$ and $z_j\in Z'$ and paths $P_j\subset G_1\cup G_2$ and $Q_j\subset H_1\cup H_2$ such that the following hold.
\stepcounter{alabel}
\begin{enumerate}[label = \textbf{\Alph{alabel}\arabic{enumi}}]
\item $P_j$ is an $x_j,x'$-path with length at most $k$ and no vertices in $V(S)\setminus\{x_j\}$, and $(S+P_j,X)$ is  $(d_0,m)$-extendable in $(G_1,G_2)$.\label{berry11}
\item $Q_j$ is a $y_{j+1},y'$-path with length at most $k$ and no vertices in $V(T)\setminus \{y_j\}$, and $(T+Q_j,Y)$ is $(d_0,m)$-extendable in $(H_1,H_2)$.\label{berry12}
\item $x'z_j$, $z_jy'\in E(G)$.\label{berry13}
\end{enumerate}
Note that, by the definition of $(d_0,m)$-extendability, for each $i$, the $i$th edge of $P_j$, counting from $x_j$ is in $G_{2-(i\mod 2)}$. Thus, by the choice of the graphs $G_1$ and $G_2$, $P_j$ is a directed path in $D$. Similarly, $Q_j$ is, when reversed, a directed path in $D$. By \ref{berry13}, and the choice of $G$, then, $P_jz_j\cev{Q_j}$ is a directed $x_jy_{j+1}$-path in $D$.

Then, $P_i$, $Q_i$ and $z_i$, $i\in I\cup \{j\}$, satisfy \ref{other1}--\ref{other3}, contradicting the maximality of $I$.
\end{proof}
\fi

\subsection{Weak connection in pseudorandom digraphs}
We now use Lemma~\ref{findapath1} to connect edges into a cycle using directed paths with, on average, length $O(1)$ (see Theorem~\ref{longpathunpicky}), rather than the (potential) length $\Omega_\eps(\log n/\lg{2})$ used in Theorem~\ref{shortpathpicky}. For this to be feasible we remove the restriction that the original edges appear in the cycle in a specified order -- this is the difference between weak and strong connection (see Definitions~\ref{weakdefn} and~\ref{strongdefn}).

Similarly to the proof of Theorem~\ref{shortpathpicky}, we prove Theorem~\ref{longpathunpicky} by first applying Lemma~\ref{divisionlemma} to find sets before, for any appropriate set of edges $E$, selecting a maximal set of paths satisfying some conditions. Lemma~\ref{findapath1} will then show that we have as many paths as possible, and thus have the cycle we need. However, for discussion it is more convenient to think about finding paths
one-by-one, each time connecting two of the edges in $E$ and moving closer to the cycle. At the start, there are many potential pairs of edges we could connect together, allowing us to apply Lemma~\ref{findapath1} with initially large sets $X'$ and $Y'$. When we have connected most of the edges into a cycle we will need to use longer paths, but we will still have, on average, used short paths.

\begin{theorem}\label{longpathunpicky}   For each $\e>0$, there exists $n_0=n_0(\e)$ such that the following holds for every $d\geq 10^{-5}$ and $n\geq n_0$. Let $D$ be an $n$-vertex $(d,\eps)$-pseudorandom digraph containing disjoint vertex sets $A$ and $V$ so that
\begin{equation}\label{late}
\frac{n\lg{2}}{\log n\cdot\lg{4}}\leq |A|\leq \frac{n(\lg{2})^3}{\log n}
\end{equation}
 and $|V|\leq |A|/\lg{6}$, and the following hold with $m=n\log^{[3]} n/d\log n$.
\begin{enumerate}
\item For each $v\in A\cup V$, $d^{\pm}(v,A)\geq 40 d\lg{2}/\lg{4}$ and $d^\pm(v,A\cup V)\leq d(\lg{2})^3$.
\item For each $U\subset A\cup V$ with $|U|=m$, $|N^\pm(U,A)|\geq (1/2+\eps/8)|A|$.
\end{enumerate}
Then, $V$ is weakly connected in $D[A\cup V]$.
\end{theorem}
\ifproofdone
\else
\begin{proof}
 Let $\alpha=\lfloor |A|/5\rfloor$, $\delta=40d\lg{2}/\lg{4}$, $\Delta=d(\lg{2})^3$ and $n'=|A\cup V|\leq 6a$.

We will now check the conditions \ref{cond0}--\ref{cond2} to apply Lemma~\ref{divisionlemma} to $D[A\cup V]$ with $n'$ in place of $n$, $\delta$, $\Delta$, $m$ unchanged, $\ell=5$, $a_1=\ldots=a_5=\alpha$ and $\eps/8$ in place of $\eps$.
For sufficiently large $n$, we have for \ref{cond0}--\ref{cond3} in turn that
\[
\left(\frac{\eps}{8}\right)^2\cdot \frac{\alpha^2}{|A|}=\Omega_\eps(\alpha)=\omega(1),
\]
\[
\left(\frac{\eps}{8}\right)^2\cdot \frac{\alpha^2}{|A|}\cdot\exp \left(\frac{\delta \alpha}{24|A|}\right)
=\Omega_\eps\left(\alpha\cdot\exp\left(\frac{d \lg{2}}{5\lg{4}}\right)\right)=\omega(\alpha)=\omega(n'),
\]
and
\[
\exp\left(\frac{\delta \alpha}{24|A|}\right)=\exp\left(\Omega\left(\frac{d\lg{2}}{\lg{4}}\right)\right)=\omega(d^2(\lg{2})^6 )=\omega(\Delta^2).
\]
Using~\eqref{late}, for \ref{cond2}, note that, as $n'\leq 6\alpha\leq 6n(\lg{2})^3/\log n$, $\log(e n'/m)=O(\log(d\lg{2}))$, and hence
\[
\left(\frac{\eps}{8}\right)^2\cdot \frac{\alpha}{10^3}=\Omega_\eps(\alpha)\overset{\eqref{late}}{=}\Omega_\eps\left(\frac{n\lg{2}}{\log n\cdot\lg{4}}\right)
=\Omega_\eps\left(m\cdot\frac{d\lg{2}}{\lg{3}\cdot\lg{4}}\right)=\omega\left(m\cdot \log\left(\frac{en'}{m}\right)\right).
\]

Having checked the appropriate conditions, by Lemma~\ref{divisionlemma}, for sufficiently large $n$, there are disjoint vertex sets $A'_1,A_2,B'_1,B_2,Z$ in $A$, each with size $\alpha$, such that
\stepcounter{alabel}
\begin{enumerate}[label = \textbf{\Alph{alabel}\arabic{enumi}}]
\item For each $v\in V(D)$ and $B\in \{A'_1,A_2,B_1',B_2,Z\}$, $d^\pm(v,B)\geq d\lg{2}/\lg{4}$.\label{unsigh1}
\item For each $U\subset V(D)$ with $|U|=m$  and $B\in \{A'_1,A_2,B_1',B_2,Z\}$, $|N^\pm(U,B)|\geq (1/2+\eps/16)|B|$.\label{unsigh2}
\end{enumerate}

Now, take an arbitrary set $E=\{\vec{y_ix_i}:i\in [\ell]\}$ of $\ell\leq |V|/2=o(a)$ vertex disjoint edges in the complete digraph with vertex set $V$. To show the lemma, it is sufficient to find a directed cycle in $D[A\cup V(E)]+E$ which contains the edges $\vec{y_1x_1},\ldots,\vec{y_\ell x_\ell}$ in any order.

Let $X=\{x_1,\ldots,x_\ell\}$ and $Y=\{y_1,\ldots y_\ell\}$. Let $G_1$ and $G_2$ be the bipartite (undirected) graphs with vertex set $A_1:=A'_1\cup X$ and $A_2$ and the edges in $D$ from $A_1$ to $A_2$ and $A_2$ to $A_1$ respectively, but without their directions. Let $H_1$ and $H_2$ be the bipartite (undirected) graphs with vertex set $B_1:=B'_1\cup Y$ and $B_2$ and the edges in $D$ from $B_2$ to $B_1$ and $B_1$ to $B_2$ respectively.
Let $G$ have vertex set $A\cup X\cup Y$ and consist of the edges in $G_1, G_2, H_1, H_2$ as well as the edges from $A_1\cup A_2$ to $Z$ and $Z$ to $B_1\cup B_2$, again without their directions.  Let $I_X$ and $I_Y$ respectively be the graphs with vertex set $X$ and $Y$ which each have no edges. Let $d_0=10$.

\begin{claim}\label{rain} For sufficiently large $n$,
$(I_X,X)$ is $(d_0,m)$-extendable in $(G_1,G_2)$ and $(I_Y,Y)$ is $(d_0,m)$-extendable in $(H_1,H_2)$.
\end{claim}
\begin{proof}[Proof of Claim~\ref{rain}] We will show that $(I_X,X)$ is $(d_0,m)$-extendable in $(G_1,G_2)$. That $(I_Y,Y)$ is $(d_0,m)$-extendable in $(H_1,H_2)$ follows similarly. Note that \ref{crack1} and \ref{crack2} in Definition~\ref{extendability} hold as $I_X$ and $I_Y$ have no edges. Furthermore, from \ref{unsigh2}, we have, for each $U\subset A_2$ with $|U|\geq m$, that
\[
|N_{G_2}(U,A_1)\setminus X|=|N_{D}^+(U,A_1')| \geq (1/2+\eps/16)|A_1'|> |A_1|/2,
\]
where we have used that $|X|=\ell=o(a)$.
Thus, as, for each $U\subset A_1$ with $|U|\geq m$, we have, from \ref{unsigh2}, that $|N_{G_1}(U,A_2)|=|N_{D}^+(U,A_2)|>|A_2|/2$, we have that \ref{crack4} holds.

Let $U\subset A_1$ with $0<|U|\leq 2m$. By \ref{unsigh1}, for sufficiently large $n$, each vertex in $U$ has at least $d\lg{2}/\lg{4}$ out-neighbours in $A_2$ in $D$.
Thus, as $A_2\cap X=\emptyset$, by \ref{pseud1} in the definition of $(d,\eps)$-pseudorandomness, we have
\begin{equation}\label{this12}
|N_{G_1}(U,A_2\setminus X)|=|N_{D}^+(U,A_2)|\geq d_0|U|.
\end{equation}
Furthermore, for sufficiently large $n$, for each $U\subset A_2$ with $0<|U|\leq 2m$, each vertex in $U$ has at least $d\lg{2}/\lg{4}$ out-neighbours in $A_1'$ in $D$ by \ref{unsigh1}. Thus, by \ref{pseud1} in the definition of $(d,\eps)$-pseudorandomness, we have
\begin{equation}\label{that12}
|N_{G_2}(U,A_1\setminus X)|=|N_{D}^+(U,A_1')|\geq d_0|U|.
\end{equation}
In combination, \eqref{this12} and \eqref{that12} show that \ref{crack3} holds to complete the proof of the $(d_0,m)$-extendability of $(I_X,X)$ in $(G_1,G_2)$.
\end{proof}

We will now cover the edges $\vec{y_ix_i}$, $i\in [\ell]$, using as few directed paths as possible, subject to some conditions (\ref{cat1}--\ref{cat4} below). We will then use Lemma~\ref{findapath1} to show that in fact we have one directed path. Applying Lemma~\ref{findapath1} again will then allow us to complete this path into a cycle.

To govern the length of the covering paths, we use a function $g$ defined as follows. For each $r\in [\ell]$,
let
\begin{equation}\label{gdefn}
g(r)=\sum_{i=1}^{r}\left\lceil\log \left.\left(\frac{4m}{\ell+1-i}\right)\right/\log (d_0-1)\right\rceil.
\end{equation}
As before, let $I_X$ and $I_Y$ be the graphs with no edges and vertex sets $X$ and $Y$ respectively.

Now, for the smallest possible $r\in [\ell]$, find vertex disjoint directed paths $R_i$, $i\in [r]$, in $D[A\cup V(E)]+E$ satisfying the following properties.
\stepcounter{alabel}
\begin{enumerate}[label = \textbf{\Alph{alabel}\arabic{enumi}}]
\item Each edge $\vec{y_ix_i}$ appears in some path $R_j$, $j\in [r]$.\label{cat1}
\item In total, the paths $R_i$, $i\in [r]$, have length at most $\ell+4\cdot g(\ell-r)$ and contain at most $\ell-r$ vertices in $Z$.\label{cat2}
\item Each path $R_i$, $i\in [r]$, starts with some vertex $y_j$ and ends with some vertex $x_{j'}$.\label{cat3}
\item Letting $P$ and $Q$ be the graphs of the edges in the paths $R_i$ which appear (without their directions) in $G_1\cup G_2$ and $H_1\cup H_2$ respectively, $(I_X+P,X)$ is $(d_0,m)$-extendable in $(G_1,G_2)$ and $(I_Y+Q,Y)$ is $(d_0,m)$-extendable in $(H_1,H_2)$. \label{cat4}
\end{enumerate}
Note that the $\ell$ paths consisting of just the edges $\vec{y_ix_i}$, $i\in [\ell]$, satisfy these properties, so such an $r$ and such paths $R_i$, $i\in [r]$, exist.

We will show, by contradiction, that $r=1$. Let us assume then that $r\geq 2$. Let $r'=\lfloor r/2\rfloor\geq 1$. Let $X'$ be a set of $r'$ end vertices of some of the paths $R_i$, $i\in [r]$, and let $Y'$ be a set of $r'$ start vertices of some of the paths $R_i$, $i\in [r]$, so that no path $R_i$, $i\in [r]$, has a vertex in both $X'$ and $Y'$. This is possible as $2r'\leq r$.
Note that, by \ref{cat3}, $X'\subset X$ and $Y'\subset Y$. Furthermore, each vertex in $X'$ appears only in some edge in $E$ in the paths $R_i$, $i\in [r]$, and therefore has degree 0 in $I_X+P$. Similarly, each vertex in $Y'$ has degree 0 in $I_Y+Q$.

We will apply Lemma~\ref{findapath1} to $I_X+P$ and $I_Y+Q$, so we need a bound on their size, which we get from the following claim.

\begin{claim}\label{overdue} We have
\[
g(\ell)=\sum_{i=1}^\ell\left\lceil\log \left.\left(\frac{4m}{\ell+1-i}\right)\right/\log (d_0-1)\right\rceil=o(a).
\]
\end{claim}
\begin{proof}[Proof of Claim~\ref{overdue}]
For any positive integer $s\geq 2$ and each $i\in [\ell]$, if $s=\left\lceil\log \left.\left(4m/(\ell+1-i)\right)\right/\log (d_0-1)\right\rceil$, then
\[
(d_0-1)^{s-1}< \frac{4m}{\ell+1-i}\leq (d_0-1)^{s},
\]
and thus
\[
\ell+1-\frac{4m}{(d_0-1)^{s-1}}<i\leq \ell+1-\frac{4m}{(d_0-1)^{s}}.
\]
Certainly, there are at most $1+4m/(d_0-1)^{s-1}$ integers $i$ which satisfy this.

Thus, we have, as $\ell\leq |V|/2\leq |A|/\lg{6}=o(a)$, that
\begin{align*}
g(\ell)\overset{\eqref{gdefn}}{=}\sum_{i=1}^\ell\left\lceil\left.\log \left(\frac{4m}{\ell+1-i}\right)\right/\log (d_0-1)\right\rceil
&\leq \ell+\sum_{s=2}^{\lceil\log (4m)/\log (d_0-1)\rceil}\left(1+\frac{4m}{(d_0-1)^{s-1}}\right)\cdot s \\
&\leq \ell+\left\lceil\frac{\log(4m)}{\log(d_0-1)}\right\rceil^2+ 4m\cdot \sum_{s=2}^{\infty}\frac{s}{(d_0-1)^{s-1}}\\
&\leq \ell+\log^2n+4m\cdot O(1)=o(a).\hfill\qedhere
\end{align*}
\end{proof}

As $\ell=o(a)$, by Claim~\ref{overdue} and \ref{cat2}, we have
\[
|I_X+P|+|I_Y+Q|\leq 3\ell+\sum_{i=1}^r 4\left\lceil\log \left.\left(\frac{4m}{\ell+1-i}\right)\right/\log (d_0-1)\right\rceil=o(a).
\]

Let $Z'=Z\setminus (\cup_{i\in [r]}V(R_i))$, noting that, by \ref{cat2}, \ref{unsigh2} and the definition of $G$, and as $\ell=o(a)$, we have, for sufficiently large $n$, that
\begin{enumerate}[label = \bfseries \ref{unsigh2}']
\item For each $U\subset A_1\cup A_2$ and $U\subset B_1\cup B_2$ with $|U|=m$, $|N_G(U,Z')|\geq (1/2+\eps/16)|Z|-\ell>|Z'|/2$.\label{newun}
\end{enumerate}

Let $k'=\lceil\log(4m/r)/\log(d_0-1)\rceil$, so that
\begin{equation}\label{conan}
g(\ell-r+1)=g(\ell-r)+k',
\end{equation}
and $|X'|=|Y'|=r'\geq r/2 \geq 2m/(d_0-1)^{k'}$.
Thus, as $d_0m=o(a)$ by \eqref{late}, by Lemma~\ref{findapath1}, there are some $j,j'\in [\ell]$, vertices $x'\in A_1\cup A_2$ and $y'\in B_1\cup B_2$, and paths $P'\subset G_1\cup G_2$ and $Q'\subset H_1\cup H_2$ such that the following hold.
\stepcounter{alabel}
\begin{enumerate}[label = \textbf{\Alph{alabel}\arabic{enumi}}]
\item $x_j\in X'$ and $y_{j'}\in Y'$.\label{bear0}
\item $P'$ is an $x_j,x'$-path with length at most $k'$, no vertices in $(X\cup V(P))\setminus\{x_j\}$, and for which $(I_X+P+P',X)$ is  $(d_0,m)$-extendable in $(G_1,G_2)$.\label{bear1}
\item $Q'$ is a $y_{j'},y'$-path with length at most $k'$ and no vertices in $(Y\cup V(Q))\setminus \{y_{j'}\}$, and $(I_Y+Q+Q',Y)$ is $(d_0,m)$-extendable in $(H_1,H_2)$.\label{bear2}
\item $N_G(x')\cap N_G(y')\cap Z'\neq \emptyset$.\label{bear3}
\end{enumerate}
Note that, by definition of the $(d_0,m)$-extendability, the $i$th edge of $P'$, counting from $x_j$ is in $G_{2-(i\mod 2)}$. Thus, by the choice of the graphs $G_1$ and $G_2$, $P'$ is a directed $x_j,x'$-path in $D$. Similarly, $Q'$ is, when reversed, a directed $y_{j'},y'$-path in $D$. Using \ref{bear3}, and noting that, by the choice of $G$,
\[
N_G(x')\cap N_G(y')\cap Z'=N^+_{D}(x')\cap N^-_D(y')\cap Z',
\]
select a vertex $z\in Z'$ such that $P'z\cev{Q'}$ is a directed $x_jy_{j'}$-path in $D$.

Note that, by \ref{bear0} and the choice of $X'$ and $Y'$, $x_j$ is the end vertex of a path different to the path of which $y_{j'}$ is the start vertex. Assume, then, by relabelling if necessary, that $x_j$ is the end vertex of $R_{r-1}$ and $y_{j'}$ is the start vertex of $R_{r}$. We will show that the $r-1$ paths $R_i'=R_i$, $i\in [r-2]$, and $R'_{r-1}=R_{r-1}P'z\cev{Q'}R_{r}$ satisfy \ref{cat1}--\ref{cat4}, contradicting the definition of $r$.

By \ref{bear1} and \ref{bear2}, the choice of $Z'$, and as the paths $R_i$, $i\in [r]$, are vertex disjoint, the paths $R_i'$, $i\in [r-1]$, are vertex disjoint. By \ref{cat1} for the paths $R_i$, $i\in [r]$, and as the paths $R'_i$ contain the paths $R_i$, \ref{cat1} holds for the paths $R'_i$, $i\in [r-1]$.
Each path $R_i'$ shares a start vertex with some path $R_{i'}$ and an end vertex with some (potentially different) path $R_{i''}$, and therefore, as \ref{cat3} holds for the paths $R_i$, $i\in [r]$, \ref{cat3} holds for the paths $R_i'$, $i\in [r-1]$.

Note that $P+P'$ and $Q+Q'$ are  exactly the graphs of edges in the paths $R_i'$ which appear (without direction) in $G_1\cup G_2$ and $H_1\cup H_2$ respectively. Thus, \ref{cat4} holds for the paths $R_i'$, $i\in [r-1]$, by \ref{bear1} and \ref{bear2}.

The paths $R_i'$, $i\in [r-1]$, contain one additional vertex in $Z$ compared to the paths $R_i$, $i\in [r]$, so that, in total, they have at most $\ell-(r-1)$ vertices in $Z$ by \ref{cat2} for the paths $R_i$. As $P'$ and $Q'$ have length at most $k'$, we have, by \ref{cat2} for the paths $R_i$ again, that the paths $R_i'$ have total length at most
\[
\ell+4g(\ell-r)+2k'+2\leq \ell+4g(\ell-r)+4k'\overset{\eqref{conan}}\leq \ell+4g(\ell-r+1).
\]
Therefore, \ref{cat2} holds for the paths $R_i'$, $i\in [r-1]$. This completes the proof that  \ref{cat1}--\ref{cat4} hold for the paths $R_i'$, $i\in [r-1]$, contradicting the choice of $r$.

Therefore, we have that $r=1$. That is, with relabelling, there is a single path $R$ in $D[A\cup V(E)]+E$ containing each edge in $E$, with start vertex $y_1$ and end vertex $x_2$, with length at most $o(a)$ (using Claim~\ref{overdue}) and at most $\ell$ vertices in $Z$ such that the following holds. If $R'$ and $R''$ are the graphs of the edges in the path $R$ which appear (without direction) in $G_1\cup G_2$ and $H_1\cup H_2$ respectively, then $(I_X+R',X)$ and $(I_Y+R'',Y)$ are $(d_0,m)$-extendable in $(G_1,G_2)$ and $(H_1,H_2)$ respectively.

Let $Z''=Z\setminus V(R)$, and note that, as $|Z\cap V(R)|\leq \ell=o(a)$, for sufficiently large $n$, by \ref{unsigh2}, we have that
\begin{enumerate}[label = \bfseries \ref{unsigh2}'']
\item For each $U\subset A_1\cup A_2$ and $U\subset B_1\cup B_2$ with $|U|=m$, $|N_G(U,Z'')|\geq (1/2+\eps/16)|Z|-\ell>|Z''|/2$.\label{newun2}
\end{enumerate}

Note that, $R'$ is a collection of paths in $G_1\cup G_2$, where each path contains exactly one vertex in $X$, which is furthermore one of its endpoints. The other endpoint of such a path has an in- or out-neighbour in $Z''$ in $R$. Similarly, $R''$ is a collection of paths in $H_1\cup H_2$, where each path contains exactly one vertex in $Y$, which is one of its endpoints. The other endpoint of such a path has an in- or out-neighbour in $Z''$ in $R$.

Now, $R$ has the start vertex $y_1$ and contains the edge $\vec{y_1x_1}$, so $y_1$ is not in $R''$. Similarly, $R$ has the end vertex $x_2$ and contains the edge $\vec{y_2x_2}$, so $x_2$ is not in $R$. Therefore $x_2$ and $y_1$ have degree 0 in $I_X+R'$ and $I_Y+R''$, respectively.

Let $k=\left\lceil\log (2m)/\log (d_0-1)\right\rceil$, so that $m/(d_0-1)^k\leq 1$. By Claim~\ref{overdue}, \ref{cat2} and \ref{cat4}, and as $d_0m=o(a)$, for sufficiently large $n$, by Lemma~\ref{findapath1} applied with $X'=\{x_2\}$ and $Y'=\{y_1\}$, there are vertices $x''\in A_1\cup A_2$ and $y''\in B_1\cup B_2$
and paths $P''\subset G_1\cup G_2$ and $Q''\subset H_1\cup H_2$ such that the following hold.
\stepcounter{alabel}
\begin{enumerate}[label = \textbf{\Alph{alabel}\arabic{enumi}}]
\item $P''$ is an $x_2,x''$-path with length at most $k$ and no vertices in $V(R)\setminus\{x_2\}$, and for which $(I_X+R'+P'',X)$ is  $(d_0,m)$-extendable in $(G_1,G_2)$.\label{berry1}
\item $Q''$ is a $y_{1},y''$-path with length at most $k$ and no vertices in $V(R)\setminus \{y_1\}$, and $(I_Y+R''+Q'',Y)$ is $(d_0,m)$-extendable in $(H_1,H_2)$.\label{berry2}
\item $N_G(x'')\cap N_G(y'')\cap Z'\neq \emptyset$.\label{berry3}
\end{enumerate}
Note that, by the definition of $(d_0,m)$-extendability, the $i$th edge of $P''$, counting from $x_2$ is in $G_{2-(i\mod 2)}$. Thus, by the choice of the graphs $G_1$ and $G_2$, $P''$ is a directed $x_2,x''$-path in $D$. Similarly, $Q''$ is, when reversed, a directed $y'',y_1$-path in $D$. Using \ref{berry3}, and noting that, by the choice of $G$,
\[
N_G(x'')\cap N_G(y'')\cap Z''=N^+_{D}(x'')\cap N^-_D(y'')\cap Z'',
\]
select a vertex $z'\in Z''$ such that $P''z'\cev{Q''}$ is a directed $x_2y_1$-path in $D[A\cup V]$.

Therefore, $RP''z'\cev{Q''}$ is a directed cycle in $D[A\cup V]+E$ containing each edge in $E$, as required.
\end{proof}
\fi



\section{Covering vertices with directed paths}\label{seccover}
 To cover most of the vertices with few directed paths in a pseudorandom digraph, we divide the vertex set into random sets using the local lemma and then find matchings between the sets. It would be nice to do this in one application of  Lemma~\ref{divisionlemma}, however this is not possible. Essentially, this would attempt to track the in- and out-degrees of all vertices into each random set, which is too much for our methods. To see this, consider the following. We wish to cover at least $n/2$ vertices in a pseudorandom digraph $D$ using paths of length $k=\Theta(\log n\cdot \log^{[5]}n/\log^{[2]}n)$, and therefore we take sets with size $\Theta(n/k)$ and find matchings between them. Looking at~\ref{cond3}, to apply Lemma~\ref{divisionlemma} to $V(D)$ in $D$ to get a partition including a set of size $\Theta(n/k)$, we need to have $\exp(\Theta(\delta(D)/k))\geq (\Delta(D))^2$. However, if $D$ is $(d,\eps)$-pseudorandom with $d=1$ (as we must consider), so that $\delta(D)=\Theta(\log n)$ and $\Delta(D)=\Theta(\log n)$, this corresponds to $\exp(\Theta(\log^{[2]}n/\log^{[5]}n))\geq \Theta(\log^2n)$. Thus, we cannot apply Lemma~\ref{divisionlemma} in this manner. There is also a very similar issue that the number of sets in the partition will be too large for \ref{cond3} if we take $k$ sets of such a size.

 However, these calculations are not far from working. Considering still $d=1$, the calculations above suggest we could use Lemma~\ref{divisionlemma} to split the digraph into sets with size $\Theta(n\lg{2}/k)$, as then $\exp(\Theta(\delta(D)\cdot \lg{2}/k))\geq (\Delta(D))^2$ can hold. By that lemma, this would give sets whose induced digraph in $D$ has maximum degree at most $\Theta(\Delta(D)\cdot \lg{2}/k)=O(\lg{2})$ and minimum degree at least $\Theta(\delta(D)\cdot \lg{2}/k)=\Theta((\lg{2})^2/\lg{5})$.
If we then apply Lemma~\ref{divisionlemma} to such an induced digraph, some $D'$ say, to divide into $\lg{2}$ different sets, \ref{cond3} requires that $\exp(\Theta(\delta(D')/\lg{2}))\geq (\Delta(D'))^2$, which now (rather easily) does hold. Doing this to the digraph induced on each set from the first application of Lemma~\ref{divisionlemma} then gives $k$ sets. Dividing in two stages also solves the similar issue with the number of sets in the partition being too large.

This double application lies behind our proof of Lemma~\ref{initialpaths}, though the proof is (only slightly) more complicated as for each vertex we need to track its in- and out-degree into the sets we want to match it with.

\begin{lemma}\label{initialpaths} For each $\e>0$, there exists $n_0=n_0(\e)$ such that the following holds for every $d\geq 10^{-5}$ and $n\geq n_0$. Suppose an $n$-vertex $(d,\e)$-pseudorandom digraph $D$ contains a set $B$ with $|B|\geq n/2$ such that the following hold with $m=n\lg{3}/d\log n$.
\stepcounter{alabel}
\begin{enumerate}[label = \textbf{\Alph{alabel}\arabic{enumi}}]
\item For each $v\in V(D)$, $d^\pm(v,B)\geq d\log n/8$.\label{das}
\item For each $U\subset V(D)$ with $|U|=m$, $|N^\pm(U,B)|\geq (1/2+\eps/2)|B|$.\label{dasother}
\end{enumerate}
Then, there is a collection of at most $n\lg{2}/(20\log n\cdot\lg{5})$ directed paths (with single vertices permitted) which partition $B$, so that any vertex $v\in V(D)$ has at most $d(\lg{2})^2$ in- or out-neighbours among their start and end vertices.
\end{lemma}
\ifproofdone
\else
\begin{proof}
Let $\ell=n\lg{2}/(50\log n\cdot\lg{5})$ and $k=\lfloor |B|/\ell\rfloor$. Noting that $k=\Theta(\log n\cdot\lg{5}/\lg{2})$, take integers $r$ and $k_1,\ldots,k_r$ such that $\lg{2}\leq k_i\leq 2\lg{2}$ and $\sum_{i\in [r]}k_i=k$, and note that, for sufficiently large $n$, $r\leq \log n$.

We will now check that the conditions \ref{cond0}--\ref{cond2} hold for an application of Lemma~\ref{divisionlemma} (using $r$ in the place of $\ell$ where it appears in that lemma, and $\eps/2$ in place of $\eps$) with $a_i=k_i\ell$ for each $2\leq i\leq r$ and $a_1=|B|-a_2-\ldots-a_r\geq k_1\ell$. Note that if these conditions hold for $k_1\ell$, then they also hold for $a_1\geq k_1\ell$.

Let $\Delta=10^6d\log n$, so that from the definition of a $(d, \eps)$-pseudorandom digraph, we have $\Delta^\pm(D)\leq \Delta$. Let $\delta=d\log n/8$, so that, for each $v\in V(D)$, $d^\pm(v,B)\geq \delta$ by \ref{das}. For \ref{cond0}, note that, for each $i\in [r]$,
\[
\left(\frac{\eps}{2}\right)^2\cdot \frac{(k_i\ell)^2}{|B|}=\Omega_\eps\left(\frac{n}{\log^2 n}\right)=\omega(r^3).
\]
Now, for each $i\in [r]$, note that
\begin{align}\label{crackle}
\frac{\delta (k_i\ell)}{24|B|}\geq\frac{d\log n\cdot(k_i\ell)}{192n}
=\Omega\left(\frac{dk_i\lg{2}}{\lg{5}}\right)=\omega\left( d\lg{2}\right).
\end{align}
For \ref{cond1}, then, for each $i\in [r]$ and sufficiently large $n$, we have
\[
\left(\frac{\eps}{2}\right)^2 \cdot\frac{(k_i\ell)^2}{|B|}\cdot\exp\left(\frac{\delta (k_i\ell)}{24|B|}\right)
\overset{\eqref{crackle}}{\geq} \frac{\eps^2\ell^2}{4n}\cdot \exp\left(\omega\left(d\lg{2}\right)\right)
= \Omega_\eps\left(\frac{n}{\log^2 n}\cdot \log^6 n\right)=\omega(nr^3).
\]
Furthermore, for \ref{cond3}, for each $i\in [r]$, as $r\leq \log n$ and $\Delta=10^6 d\log n$, we have
\[
\exp\left(\frac{\delta(k_i\ell)}{24|B|}\right)\overset{\eqref{crackle}}{=}\exp\left(\omega(d\lg{2})\right)=\omega(r\Delta^2).
\]
Finally, as $\log (en/m)=O(\log(d\log n))$, we have, for each $i\in [r]$,
\[
\left(\frac{\eps}{2}\right)^2\cdot \frac{k_i\ell}{10^3}=\Omega_\eps\left( \frac{n(\lg{2})^2}{\log n\cdot\lg{5}}\right)=\Omega_\eps\left(m\cdot \frac{d(\lg{2})^2}{\lg{3}\cdot\lg{5}}\right)=\omega\left( m\log\left(\frac{en}{m}\right)\right).
\]

Having checked the appropriate conditions, for sufficiently large $n$, by Lemma~\ref{divisionlemma}, we can take disjoint sets $B_1,\ldots,B_r$ in $B$ so that
\stepcounter{alabel}
\begin{enumerate}[label = \textbf{\Alph{alabel}\arabic{enumi}}]
\item For each $2\leq i\leq r$, $|B_i|=k_i\ell$, and $k_1\ell\leq |B_1|=|B|-a_2-\ldots-a_r\leq (k_1+1)\ell$.\label{peace}
\item For each $i\in [r]$ and $v\in V(D)$, \label{peace2}
\[
\frac{\delta k_i\ell}{4|B|}\leq  d^\pm(v,B_i)\leq \frac{4\Delta (k_i+1)\ell}{|B|}\leq \frac{16\Delta k_i\ell}{n}=o(d(\lg{2})^2).
\]
\item For each $U\subset B$ with $|U|=m$, we have $|N^\pm(U,B_i)|\geq (1/2+\eps/4)|B_i|$.\label{peace3}
\end{enumerate}
\noindent
The directed paths we find will all start and end in $B_1\cup B_r$, so that, by \ref{peace2}, for sufficiently large $n$, every vertex in $V(D)$ has at most $d(\lg{2})^2$ in- and out-neighbours among these vertices.

Let $D_1=D[B_1\cup B_2]$, $D_r=D[B_{r-1}\cup B_{r}]$, and, for each $2\leq i\leq r-1$, let $D_i=D[B_{i-1}\cup B_i\cup B_{i+1}]$. For each $i\in [r]$, let
\begin{equation}
n_i=|D_i|\leq 5k_i\ell+\ell\leq \frac{n(\lg{2})^2}{\log n}\leq m\cdot d(\lg{2})^2.\label{nieq}
\end{equation} Let $\bar{\Delta}=d(\lg{2})^2/2$, so that, by \ref{peace2}, for sufficiently large $n$, $\Delta(D_i)\leq \bar{\Delta}$ for each $i\in [r]$.

Let $d_0=d\lg{2}/\lg{4}$, so that $\exp(d_0)=\omega(k_i\bar{\Delta}^2)$ for each $i\in [r]$. Furthermore, for each $i\in [r]$, let $\delta_i=\delta k_i\ell/4|B|$, so that, by \ref{peace2}, for each $v\in V(D_i)$,
\begin{equation}\label{bread}
d^\pm(v,B_i)\geq \frac{\delta k_i\ell}{4|B|}=\delta_i\geq \frac{(d \log n)\cdot k_i\cdot \ell}{32n}\geq
\frac{dk_i\lg{2}}{10^4\lg{5}}=\omega(k_id_0).
\end{equation}

We will now check the conditions \ref{cond0}--\ref{cond2} to apply Lemma~\ref{divisionlemma} with $\eps/4$ in place of $\eps$ to the set $B_i$ in the digraph $D_i$, for each $i\in [r]$, to find, in $B_i$, $k_i$ disjoint sets with size $\ell$. Note that, for each $i\in [r]$, $|B_i|\leq k_i\ell+\ell$.
First, for \ref{cond0}, we have, for each $i\in [r]$,
\[
\left(\frac{\eps}{4}\right)^2\cdot \frac{\ell^2}{k_i\ell+\ell}=\Omega_\eps\left(\frac{\ell}{k_i}\right)=\omega\left(\frac{n}{\log^2 n}\right)=\omega(k_i^3).
\]
Secondly, for \ref{cond1}, we have, for each $i\in [r]$,
\begin{align*}
\left(\frac{\eps}{4}\right)^2\cdot\frac{\ell^2}{k_i\ell+\ell}\cdot\exp\left(\frac{\ell\delta_i}{24(k_i\ell+\ell)}\right)
=\Omega_\eps\left(\frac{\ell}{k_i}\cdot\exp\left(\frac{\delta_i}{48k_i}\right)\right)
&\overset{\eqref{bread}}{=}\Omega_\eps\left( \frac{n}{\log n\cdot\lg{5}}\exp\left(d_0\right)\right)
\\
&=\Omega_\eps\left(\frac{n}{\log n}\cdot (\lg{2})^6\right)\overset{\eqref{nieq}}{=}\omega(n_ik_i^3).
\end{align*}
Furthermore, for \ref{cond3}, we have, for each $i\in [r]$,
\[
\exp\left(\frac{\ell \delta_i}{24(k_i\ell+\ell)}\right)\overset{\eqref{bread}}{=}\Omega\left(\exp\left(d_0\right)\right)=\omega(k_i\bar{\Delta}^2).
\]
Finally, for \ref{cond2}, for each $i\in [r]$, by \eqref{nieq}, we have $\log(en_i/m)=O( \log (d\lg{2}))$, we have
\[
\left(\frac{\eps}{4}\right)^2\cdot \frac{\ell}{10^3}=\Omega_\eps\left(\frac{n\lg{2}}{\log n\cdot\lg{5}}\right)= \Omega_\eps\left(m\cdot \frac{d\lg{2}}{\lg{5}\cdot\lg{3}}\right)
=\omega\left(m\log\left(\frac{en_i}{m}\right)\right).
\]

Having checked the appropriate conditions, for sufficiently large $n$, by Lemma~\ref{divisionlemma} applied to the set $B_i$ in the digraph $D_i$, for each $i\in [r]$, we can find in $B_{i}$ disjoint sets $B_{i,1}, \ldots, B_{i,k_i}$ so that the following hold.
\stepcounter{alabel}
\begin{enumerate}[label = \textbf{\Alph{alabel}\arabic{enumi}}]
 \item For each $i\in [r]$ and $j\in [k_i]$, $|B_{i,j}|=\ell$.\label{sad1}
 \item For each $i\in [r]$, $j\in [k_i]$ and $v\in V(D_i)$, $d^\pm(v,B_{i,j})\geq \delta_i\ell/4(k_i\ell+\ell)\geq d_0$ (using \eqref{bread}).\label{sad2}
 \item For each $i\in [r]$, $j\in [k_i]$ and $U\subset V(D_i)$ with $|U|=m$, we have $|N^\pm(U,B_{i,j})|\geq (1/2+\eps/8)|B_{i,j}|$.\label{sad3}
 \end{enumerate}
Note that, for $2\leq i\leq r$, the sets  $B_{i,1}, \ldots, B_{i,k_i}$ partition $B_i$, and the sets $B_{1,1},\ldots,B_{1,k_1}$ cover all but at most $\ell$ vertices in $B_1$.

Recall that $\sum_{i\in [r]}k_i=k$. Relabelling the sets $B_{1,1},\ldots,B_{1,k_1},B_{2,1},\ldots,B_{2,k_2},\ldots,B_{r,1},\ldots,B_{r,k_r}$ as $C_1,\ldots, C_k$ respectively, from \ref{sad1}--\ref{sad3}, we have that $|C_i|=\ell$ for each $i\in [k]$ and the following hold.
\stepcounter{alabel}
\begin{enumerate}[label = \textbf{\Alph{alabel}\arabic{enumi}}]
 \item For each $1\leq i \leq k-1$ and $v\in C_i$, $d^+(v,C_{i+1})\geq d_0$.\label{dusk1}
 \item For each $2\leq i \leq k$ and $v\in C_i$, $d^-(v,C_{i-1})\geq d_0$.\label{dusk2}
 \item For each $1\leq i \leq k-1$ and $U\subset C_i$ with $m\leq |U|\leq \lceil \ell/2\rceil$, we have $|N^+(U,C_{i+1})|\geq (1/2+\eps/8)\ell\geq |U|$.\label{dusk3}
 \item For each $2\leq i \leq k$ and $U\subset C_i$ with $m\leq |U|\leq \lceil \ell/2\rceil$, we have $|N^-(U,C_{i-1})|\geq (1/2+\eps/8)\ell\geq |U|$.\label{dusk4}
 \end{enumerate}

By \ref{dusk1}, \ref{dusk2} and \ref{pseud1} in the definition of $(d,\eps)$-pseudorandomness, the following hold.
\begin{enumerate}[label = {\bfseries \Alph{alabel}\arabic{enumi}'}]
\item For each $1\leq i \leq r-1$ and $U\subset C_i$ with $|U|\leq m$, $|N^+(U,C_{i+1})|\geq |U|$.\label{dusk12}
\item For each $2\leq i \leq r$ and  $U\subset C_i$ with $|U|\leq m$, $|N^-(U,C_{i-1})|\geq |U|$.\label{dusk22}
\end{enumerate}

Thus, for each $i\in [r-1]$, by \ref{dusk12}, \ref{dusk22}, \ref{dusk3}, \ref{dusk4} and Proposition~\ref{matchymatchy} applied to the bipartite graph between $C_i$ and $C_{i+1}$ with (undirected) edges those directed from $C_i$ to $C_{i+1}$ in $D$,  there is a matching from $C_i$ into $C_{i+1}$ in $D$. Combining such matchings gives $\ell$ vertex disjoint paths covering $C_1\cup \ldots\cup C_{k}$. These paths start in $B_1$ and end in $B_r$, and cover all the vertices in $B$ except for $|B_1|-k_1\ell\leq \ell$ vertices in $B_1$. Taking these paths with the uncovered vertices in $B_1$, to get at most $2\ell$ paths, thus gives the required partition of $B$.
\end{proof}
\fi

We wish to cover most of our digraph with few paths, all of which end in a certain subset $B_2$, in order to have \ref{D4} in the definition of a good partition. To do this, we use Lemma~\ref{initialpaths} to cover the vertices outside of $B_2$ with few paths, and then use Lemma~\ref{divisionlemma} in the same way as before to find directed paths covering $B_2$. By matching the end vertices of the paths outside $B_2$ into some of the start vertices of the paths in $B_2$, and similarly attaching the start vertices of the paths outside $B_2$ into some of the end vertices of the paths in $B_2$, we will get a set of paths which start and end in $B_2$.

\begin{lemma}\label{finalpaths} For each $\e>0$, there exists $n_0=n_0(\e)$ such that the following holds for every $d\geq 10^{-5}$ and $n\geq n_0$. Suppose an $n$-vertex $(d,\e)$-pseudorandom digraph $D$ contains disjoint sets $B_1$ and $B_2$ such that the following hold with
\[
d_0=\frac{d\lg{2}}{\lg{4}},\;\;\;\; \ell=\frac{n\lg{2}}{\log n\cdot\lg{5}},\;\;\;\;\text{ and }\;\;\;\;k=\lg{2}.
\]
\begin{enumerate}[label= {\bfseries P\arabic{enumi}}]
\item $|B_1|\geq n/2$ and $|B_2|=k\ell$. \label{house1}
\item For each $v\in V(D)$, $d^\pm(v,B_1)\geq d\log n/8$ and, for each $U\subset V(D)$ with $|U|=m$, $|N^\pm(U,B_1)|\geq (1/2+\eps/2)|B_1|$. \label{house2}
\item For each $v\in V(D)$, $4d_0k\leq d^\pm(v,B_2)\leq d(\lg{2})^2$ and, for each $U\subset V(D)$ with $|U|=m$, $|N^\pm(U,B_2)|\geq (1/2+\eps/4)|B_2|$. \label{house3}
\end{enumerate}
Then, for any set $V\subset V(D)\setminus (B_1\cup B_2)$ with $|V|\leq \eps |B_1|/20-2$ and $u,v\in V(D)\setminus(V\cup B_1\cup B_2)$, there is a set of at most $\ell$ directed paths with length at least 1 which partition $V\cup B_1\cup B_2\cup\{u,v\}$, each start and end in $B_2$, and one of which contains $\vec{uv}$.
\end{lemma}
\ifproofdone
\else
\begin{proof}
Let $B=V\cup B_1\cup\{u,v\}$ and note, that, by \ref{house2}, as $|V\cup\{u,v\}|\leq \eps |B|/20$, we have the following.
\stepcounter{alabel}
\begin{enumerate}[label = \textbf{\Alph{alabel}\arabic{enumi}}]
\item \label{spart1} For each $U\subset V(D)$ with $|U|=m$, we have $|N^{\pm}(U,B)|\geq (1/2+\eps/2)|B_1|\geq (1/2+\eps/4)|B|$.
\item \label{spart2} For each $v\in V(D)$, we have $d^\pm(v,B)\geq d\log n/8$.
\end{enumerate}
Thus, by \ref{spart1}, \ref{spart2}, \ref{house1} and Lemma~\ref{initialpaths}, there is a collection of directed paths $P_1,\ldots,P_r$, for some $r\leq \ell/20$, in $D$ which partition $B$ (allowing single vertices as paths) and so that every vertex in $D$ has at most $d(\lg{2})^2$ in- and out-neighbours among their start and end
vertices. Let $P_{r+1}$ be the path with length 1 consisting of the edge $\vec{uv}$.
For each $i\in [r+1]$, label vertices so that $P_i$ is an $x_i,y_i$-path (noting that we may have $x_i=y_i$). Let $X=\{x_1,\ldots,x_{r+1}\}$ and $Y=\{y_1,\ldots,y_{r+1}\}$, and let $D'$ be the digraph on the vertex set $B_2\cup X\cup Y$ whose edges are the edges in $D[B_2\cup X\cup Y]$ with at least one vertex in $B_2$.

Note that every vertex has at most $(\lg{2})^2+2$ in- or out-neighbours in $X\cup Y$. Combining this with~\ref{house3}, letting $\bar{\Delta}=d(\lg{2})^3$, we have, for sufficiently large $n$, that $\Delta^\pm(D')\leq \bar{\Delta}$.  Let $k_1=\lg{2}$, so that, $\lg{2} \leq k_1\leq 2\lg{2}$, and, by definition of $k$, $k_1=k$.
Note that, for each $v\in V(D')$, $d^\pm(v,B_2)\geq 4k_1d_0$ by \ref{house3}.
Let $n_1=|D'|$, so that $n_1\leq 2k_1\ell$.
The conditions \ref{cond0}--\ref{cond2} for an application of Lemma~\ref{divisionlemma} to $D'$ to partition $B_2$ into $k$ sets with size $\ell$ hold very similarly to the same conditions in the second application of this lemma in the proof of Lemma~\ref{initialpaths} (from \eqref{nieq} onwards) -- all that differs is a factor of 50 in the value of $\ell$.

Thus, for sufficiently large $n$, by Lemma~\ref{divisionlemma} there is a partition $C_1\cup\ldots\cup C_k$ of $B_2$ such that $|C_i|=\ell$ for each $i\in [\ell]$, and the following hold.
\stepcounter{alabel}
\begin{enumerate}[label = \textbf{\Alph{alabel}\arabic{enumi}}]
\item For each $i\in [\ell]$ and $U\subset V(D')$ with $|U|=m$, we have $|N^\pm(U,C_i)|\geq (1/2+\eps/8)|C_i|$.\label{raining1}
\item For each $v\in V(D')$ and $i\in [\ell]$, we have $d^\pm(v,C_i)\geq 4kd_0\cdot \ell /4k\ell= d_0$.\label{raining2}
\end{enumerate}

Similarly to the reasoning in the proof of Lemma~\ref{initialpaths}, for each $i\in [k-1]$, there is a matching from $C_i$ into $C_{i+1}$. Combine these matchings to get directed paths $Q_i$, $i\in [\ell]$, which cover $B_2$. Similarly to the reasoning in the proof of Lemma~\ref{initialpaths}, by~\ref{raining1} and \ref{raining2}, for each $U\subset X$ we have $|N^{-}(U,C_k)|\geq |U|$, and for each $U\subset Y$ we have $|N^+(U,C_1)|\geq |U|$. Note that here it is important that $|X|\leq |C_k|/2$ and $|Y|\leq |C_1|/2$. Thus, as Hall's matching condition is satisfied, we can find vertex disjoint edges $e_i$, $i\in [r+1]$, directed from $Y$ into $C_1$, and vertex disjoint edges $f_i$, $i\in [r+1]$ directed from $C_k$ into $X$. Renaming if necessary, assume that, for each $i\in [r+1]$, $x_i\in f_i$ and $y_i\in e_i$.

Note that combining the paths $Q_i$, $i\in [\ell]$, and paths $P_i+e_i+f_i$, $i\in [r+1]$, gives a collection of at most $\ell$ directed paths and cycles in $D$ which cover $V\cup B_1\cup B_2\cup\{u,v\}$ and so that each path starts and ends in $B_2$, each cycle contains some path $Q_i$, and as the edge $\vec{uv}$ is contained in $P_{r+1}$, it is contained in one path or cycle. Breaking an edge in some $Q_i$ in each cycle, gives then the required set of paths.
\end{proof}
\fi


\section{Finding an $(\ell,r)$-good partition}\label{secfindgood}
To find a good partition of a pseudorandom digraph, we first apply Lemma~\ref{divisionlemma} twice to find the necessary sets for the good partition and record the properties we get, as follows.

\begin{lemma}\label{goodpart0} For each $\e>0$, there exists $n_0=n_0(\e)$ such that the following holds for every $d\geq 10^{-5}$ and $n\geq n_0$ with
\[
k = \lg{2},\;\;\;
\ell=\frac{n\log^{[2]}n}{\log n\cdot \log^{[5]}n},\;\;\;d_0=\frac{d\lg{2}}{\lg{4}},\;\;\;
m=\frac{n\lg{3}}{d\log n},\;\;\;
\text{ and }
\;\;\;
r=\frac{n\log^{[2]}n}{\log n \cdot\log^{[6]}n}.
\]
Every $n$-vertex $(d,\e)$-pseudorandom digraph $D$ has a partition $V(D)=A\cup B_1\cup B_2\cup R_1\cup R_2\cup R_3\cup R_4$ such that the following hold.
\stepcounter{alabel}
\begin{enumerate}[label = \textbf{\Alph{alabel}\arabic{enumi}}]
\item $|A|=\eps n/40$, $|B_2|=k\ell$
and  $|R_1|=|R_2|=|R_3|=|R_4|=r$.\label{door1}
\item For each $v\in V(D)$, $d^\pm(v,A)\geq d(\log n)^{3/4}$ and, for each $U\subset V(D)$ with $|U|=m$, $|N^\pm(U,A)|\geq (1/2+\eps/2)|A|$. \label{door2}
\item For each $v\in V(D)$, $d^\pm(v,B_1)\geq d\log n/8$ and, for each $U\subset V(D)$ with $|U|=m$, $|N^\pm(U,B_1)|\geq (1/2+\eps/2)|B_1|$. \label{door3}
\item For each $v\in V(D)$, $4d_0k\leq d^\pm(v,B_2)\leq d(\lg{2})^2$ and, for each $U\subset V(D)$ with $|U|=m$, $|N^\pm(U,B_2)|\geq (1/2+\eps/4)|B_2|$. \label{door4}
\item For each $v\in B_2\cup R_1\cup R_2\cup R_3\cup R_4$ and $i\in [4]$, $40d_0\leq d^\pm(v,R_i)\leq d(\lg{2})^2$, and, for each $U\subset B_2\cup R_1\cup R_2\cup R_3\cup R_4$ with $|U|=m$, $|N^\pm(U,R_i)|\geq (1/2+\eps/4)|R_i|$. \label{door5}
\end{enumerate}
\end{lemma}
\begin{proof}
We will apply Lemma~\ref{divisionlemma} twice, again so that the second application may be with a stronger maximum degree condition.

We will check the conditions \ref{cond0}--\ref{cond2} for an application of Lemma~\ref{divisionlemma} to find a partition $V(D)=A\cup B_1\cup B_2'\cup B_3'$. Let
\begin{equation}\label{aidefn}
a_1=\frac{\eps n}{40},\;\;\;\; a_3=(1-\eps/5)k\ell,\;\;\;\; a_4=\frac{\eps k\ell}{5}+4r,\;\;\;\; \text{ and }\;\;\;\;a_2=n-a_1-a_3-a_4\geq \frac{n}{2},
\end{equation}
where the last inequality follows for sufficiently large $n$. Let $\delta=d\log n$ and $\Delta=10^6d\log n$. From the definition of $(d,\eps)$-pseudorandomness, we have that $\delta^\pm(D)\geq \delta$ and $\Delta^\pm(D)\leq \Delta$.

Note that $a_1,a_2,a_3\geq a_4$, so that we need only check \ref{cond0}--\ref{cond2} for $a_4$. First, for \ref{cond0}, we have
\begin{equation}\label{dom}
\frac{\eps^2a_4^2}{n}= \Omega_\eps\left(\frac{n}{\log^2 n}\right)=\omega(1).
\end{equation}
Note that $k\ell\delta/n=\omega(d\lg{2})$. Thus, for \ref{cond1}, we have
\[
\frac{\eps^2a_4^2}{n}\cdot \exp\left(\frac{a_4\delta}{24n}\right)
\overset{\eqref{dom}}{=} \Omega_\eps\left(\frac{n}{\log^2n}\cdot \exp\left(\frac{\eps k\ell \delta}{10^3n}\right)\right)=
\frac{n}{\log^2n}\cdot \exp\left(\omega(d\lg{2})\right)=
\omega(n).
\]
Furthermore, for \ref{cond3}, we have
\[
\exp\left(\frac{a_4\delta}{24n}\right)\geq \exp\left(\frac{\eps k\ell\delta}{10^3n}\right)=\exp\left(\omega(d\lg{2})\right)=\omega(d^2\log^2n)=\omega(\Delta^2).
\]
Finally, we have that $\log(en/m)=O(\log(d\log n))$, so that, for \ref{cond2},
\[
\frac{\eps^2a_4}{10^3}=\Omega_\eps(k\ell)=\Omega_\eps\left(\frac{n (\lg{2})^2}{\log n\cdot\lg{5}}\right)=\Omega_\eps\left(m\cdot\frac{d (\lg{2})^2}{\lg{3}\cdot \lg{5}}\right)=\omega\left(m\cdot\log\left(\frac{en}{m}\right)\right).
\]
As $a_1=\eps n/40$, $a_1\delta/4n= \eps d\log n/160=\omega(d(\log n)^{3/4})$. As $a_2\geq n/2$, $a_2\delta/4n\geq d\log n/8$.

Having checked the appropriate conditions, for sufficiently large $n$, by Lemma~\ref{divisionlemma}, $V(D)$ has a partition $A\cup B_1\cup B_2'\cup B_3'$ so that $|A|=a_1=\eps n/40$, $|B_1|=a_2$, $|B_2'|=a_3=(1-\eps/5)k\ell$, $|B_3'|=a_4=\eps k\ell/5+4r$, and \ref{door2} and \ref{door3} hold along with the following.
\stepcounter{alabel}
\begin{enumerate}[label = \textbf{\Alph{alabel}\arabic{enumi}}]
\item For each $v\in V(D)$, $a_3\delta/4n\leq d^\pm(v,B_2')\leq 4a_3\Delta/n$, and, for each $U\subset V(D)$ with $|U|=m$, $|N^\pm(U,B'_2)|\geq (1/2+\eps/2)|B_2'|$.\label{eqsub1}
\item For each $v\in V(D)$, $a_4\delta/4n\leq d^\pm(v,B_3')\leq 4a_4\Delta/n$, and, for each $U\subset V(D)$ with $|U|=m$, $|N^\pm(U,B'_3)|\geq (1/2+\eps/2)|B_3'|$.\label{eqsub2}
\end{enumerate}

Let $D'=D[B_2'\cup B_3']$. Let $n'=|D'|=a_3+a_4= k\ell+4r\leq 2k\ell$, for sufficiently large $n$, and let $ \bar{\Delta}=d(\lg{2})^2$.
Note that
\begin{equation}\label{newkevin}
\frac{4(a_3+a_4)\Delta}{n}\leq \frac{8k\ell\Delta}{n}=O\left(\frac{d(\lg{2})^2}{\lg{5}}\right)=o(\bar{\Delta}).
\end{equation}
Thus, by \ref{eqsub1} and \ref{eqsub2}, for sufficiently large $n$ we have $\Delta(D')\leq \bar{\Delta}$.
Using \eqref{aidefn}, let
\begin{equation}\label{kevin}
\delta'=\frac{a_4\delta}{4n}\geq \frac{\eps k\ell\delta}{20n}
=\Omega_\eps\left(k\ell\cdot\frac{d\lg{2}}{r\lg{6}}\right)
=\Omega_\eps\left(\frac{k\ell}{r}\cdot\frac{d_0\lg{4}}{\lg{6}}\right)
=\omega\left(\frac{k\ell}{r}\cdot d_0\right).
\end{equation}
By \ref{eqsub2}, for each $v\in V(D')$, we have $d^\pm(v,B_3')\geq \delta'$.

We will now check the conditions \ref{cond0}--\ref{cond2} to apply Lemma~\ref{divisionlemma} to $D'$ to get 4 sets, $R_1$, $R_2$, $R_3$ and $R_4$ in $B_3'$, each with size $r$.
First, for \ref{cond0}, as $n'\leq 2k\ell$, we have
\[
\left(\frac{\eps}{2}\right)^2\cdot \frac{r^2}{n'}=\Omega_\eps\left(\frac{n}{\log n}\right)=\omega(1).
\]
Next, for \ref{cond1}, we have, as $n'\leq 2k\ell$,
\[
\left(\frac{\eps}{2}\right)^2\cdot\frac{r^2}{n'}\cdot\exp\left(\frac{r\delta'}{24n'}\right)
=\Omega_\eps\left(
n'\cdot\left(\frac{r}{k\ell}\right)^2\cdot \exp\left(\frac{r\delta'}{48k\ell}\right)\right)
\overset{\eqref{kevin}}{=}\Omega_\eps\left( n'\cdot\left(\frac{1}{\lg{2}}\right)^2\cdot\exp\left(\omega(d_0)\right)\right)
=\omega(n').
\]
Furthermore, for \ref{cond3}, we have
\[
\exp\left(\frac{r\delta'}{24n'}\right)\geq \exp\left(\frac{r\delta'}{24k\ell}\right)
\overset{\eqref{kevin}}{=}\exp\left(\omega(d_0)\right)=\omega(\bar{\Delta}^2).
\]
Finally, we have that $\log(en'/m)\leq\log(2ek\ell/m)=O(\log(d\lg{2}))$, so that
\[
\left(\frac{\eps}{2}\right)^2\cdot\frac{r}{10^3}=\Omega_\eps\left(\frac{n\lg{2}}{\log n\cdot\lg{6}}\right)
=\Omega_\eps\left(m\cdot\frac{d\lg{2}}{\lg{3}\cdot\lg{6}}\right)=\omega\left(m \cdot\log\left(\frac{en'}{m}\right)\right).
\]

Having checked the appropriate conditions, by Lemma~\ref{divisionlemma} and \eqref{kevin}, for sufficiently large $n$, there are disjoint sets $R_1,R_2,R_3,R_4$ in $B_3'$ so that the following holds.
\begin{enumerate}[label = \textbf{\Alph{alabel}3}]
\item For each $v\in B_2'\cup B_3'$ and $i\in [4]$, $d^\pm(v,R_i)\geq 40d_0$, and, for each $U\subset B_2'\cup B_3'$ with $|U|=m$, $|N^\pm(U,R_i)|\geq (1/2+\eps/4)|R_i|$. \label{door51}
\end{enumerate}
Let $B_2=(B_2'\cup B_3')\setminus (R_1\cup R_2\cup R_3\cup R_4)$, so that $|B_2|=a_3+a_3-4r=k\ell$. Note that we have chosen our set sizes so that \ref{door1} holds. Note further that $B_2'\subset B_2$, so that, by \ref{eqsub1}, for each $U\subset V(D)$ with $|U|= m$,
\begin{equation}\label{newnewkevin}
|N^{\pm}(U,B_2)|\geq (1/2+\eps/2)|B_2'|=(1/2+\eps/2)(1-\eps/5)k\ell\geq (1/2+\eps/4)|B_2|.
\end{equation}
By \ref{eqsub1}, \ref{eqsub2}, \eqref{newkevin} and \eqref{kevin}, and for each $v\in V(D)$ we have, for sufficiently large $n$, $4d_0k\leq d^\pm(v,B_2)\leq \bar{\Delta}$. Therefore, in combination with \eqref{newnewkevin}, we have that \ref{door4} holds.

Note that, for each $v\in B_2\cup R_1\cup R_2\cup R_3\cup R_4=V(D')$ and $i\in [4]$, we have $d^\pm(v,R_i)\leq \bar{\Delta}=d(\lg{2})^2$. Therefore, with \ref{door51}, we have that \ref{door5} holds. This completes the proof that \ref{door1}--\ref{door5} hold, so we have found the partition as required.
\end{proof}

We now combine the work in the last few sections to find good partitions.
\begin{lemma}\label{goodpart} For each $\e>0$, there exists $n_0=n_0(\e)$ such that the following holds for every $d\geq 10^{-5}$. Every $(d,\e)$-pseudorandom digraph $D$ with at least $n_0$ vertices has a good partition.
\end{lemma}
\begin{proof} Let
\[
k = \lg{2},\;\;\;
\ell=\frac{n\log^{[2]}n}{\log n \cdot\log^{[5]}n},\;\;\;d_0=\frac{d\lg{2}}{\lg{4}},\;\;\;
m=\frac{n\lg{3}}{d\log n},\;\;\;
\text{ and }
\;\;\;
r=\frac{n\log^{[2]}n}{\log n \cdot\log^{[6]}n}.
\]
Let $D$ be an $n$-vertex $(d,\e)$-pseudorandom digraph. Letting $n$ be sufficiently large, by Lemma~\ref{goodpart0} we can find a partition $V(D)=A\cup B_1\cup B_2\cup R_1\cup R_2\cup R_3\cup R_4$ such that \ref{door1}--\ref{door5} hold.
By Theorem~\ref{shortpathpicky}, \ref{door1} and \ref{door2}, and observing that
\[
r=O\left(\frac{|A|\cdot\lg{2}}{\log n\cdot \lg{6}}\right)=o\left(\frac{|A|\cdot\lg{2}}{\log n\cdot \lg{7}}\right),
\]
we have, for sufficiently large $n$, that \ref{D3} holds in Definition~\ref{gooddefn}.
By Lemma \ref{finalpaths}, \ref{door3} and \ref{door4}, and as $|A\cup R_1\cup R_2\cup R_3\cup R_4|\leq \eps n/40+4r$, for sufficiently large $n$ we have that \ref{D4} holds in Definition~\ref{gooddefn}.

Using \ref{door5} and \ref{pseud1}, by the simple reasoning at the end of the proof of Lemma~\ref{initialpaths}, the conditions in Proposition~\ref{matchymatchy} hold for the edges directed from $R_i$ into $R_j$ for any $j\neq i$. Therefore, we can find matchings $M_1$, $M_2$ and $M_3$ from $R_2$ into $R_1$, $R_2$ into $R_3$ and $R_4$ into $R_3$ in $D$, respectively.

 Let $f:R_1\to R_4$ come from the matchings $M_1$, $M_2$, and $M_3$, and suppose each vertex $v\in R_1$ is merged into $f(v)$ in $D$ to get the digraph $D'$ (as in \ref{D5}). Let $R$ be the set of merged vertices in $D'$. For each $v\in R$, let $v^-\in R_1$ and $v^+\in R_4$ be such that $v^-$ is merged into $v^+$ to create $v$. For each $U\subset V(D')$, let $U^-=(U\setminus R)\cup\{v^-:v\in U\cap R\}$ and $U^+=(U\setminus R)\cup \{v^+:v\in U\cap R\}$. By \ref{door4} and \ref{door5}, we have that the following hold.
\stepcounter{alabel}
\begin{enumerate}[label = \textbf{\Alph{alabel}\arabic{enumi}}]
\item For each $v\in R\cup B_2$, we have
$d_{D'}^+(v,R)= d_D^+(v^+,R_1)\geq 40d_0$ and $d_{D'}^-(v,R)= d_D^-(v^-,R_4)\geq 40d_0$.\label{crack11}
\item For each $v\in R\cup B_2$, we have
$d_{D'}^+(v,R\cup B_2)= d_D^+(v^+,R_1\cup B_2)\leq 2d(\lg{2})^2$ and $d_{D'}^-(v,R\cup B_2)= d_D^-(v^-,R_4\cup B_2)\leq 2d(\lg{2})^2$. Thus, $\Delta^\pm(D'[R\cup B_2])\leq 2d(\lg{2})^2$. \label{newbie}
\item For each $U\subset R\cup B_2$, with $|U|=m$, we have
$|N_{D'}^+(U,R)|\geq |N_D^+(U^+,R_1)|\geq (1/2+\eps/8)r$ and $|N_{D'}^-(U,R)|\geq  |N_D^-(U^-,R_4)|\geq (1/2+\eps/8)r$.\label{crack21}
\end{enumerate}
Note that
\[
\ell=\frac{r\lg{6}}{\lg{5}}=o\left(\frac{r}{\lg{6}}\right).
\]
Thus, for sufficiently large $n$, for any $V\subset B_2$ with $|V|\leq 2\ell$, by Theorem~\ref{longpathunpicky}, \ref{crack11}, \ref{newbie} and \ref{crack21}, $V$ is weakly connected in $D'[R\cup V]$. Therefore, \ref{D2} and \ref{D5} hold in Definition~\ref{gooddefn}.

Therefore, \ref{D3}--\ref{D5} hold for the partition $V(D)=A\cup B_1\cup B_2\cup R_1\cup R_2\cup R_3\cup R_4$, and thus $D$ has an $(\ell,r)$-good partition.
\end{proof}

Theorem~\ref{pseudorandom} now follows immediately from Lemma~\ref{goodisgood} and Lemma~\ref{goodpart}.


\section{Pseudorandomness of random digraphs}\label{secrand}

In this section, we study the pseudorandom properties of digraphs in the random digraph process, allowing us then to apply Theorem~\ref{pseudorandom} to prove Theorem~\ref{directedresilience}. This section is organised as follows. First, in Section~\ref{sec71}, we give some simple results on maximum and minimum in- and out-degree, to later show that \ref{pseud0} resiliently holds in Definition~\ref{pseuddefn}.  Next, in Section~\ref{sec72}, with \ref{pseud3} in mind, we give a simple result concerning the edges between sets.
Then, in Section~\ref{sec73}, we prove a result showing expansion will follow from minimum degree conditions (Lemma~\ref{mindegexp}), which will allow us to show that \ref{pseud1} and \ref{pseud2} hold.
In Section~\ref{sec74}, we then study the vertices with low in- and out-degree in the digraphs early in the random digraph process. After recording together all the properties we use, in Section~\ref{sec75} we then prove the resilience of Hamiltonicity in the random digraph process needed for Theorem~\ref{directedresilience}. Finally, in Section~\ref{secnonres}, we study the limits of the resilience of Hamiltonicity to complete the proof of Theorem~\ref{directedresilience}.

It will often be convenient to show that properties are likely in $D(n,p)$, before inferring these properties are also likely in the random digraph process. Let $D_{n,M}$ be the random digraph with $n$ vertices and $M$ edges, chosen uniformly at random from all such digraphs. Note that, in the $n$-vertex random digraph process $D_0,D_1,\ldots,D_{n(n-1)}$, for each $0\leq M\leq n(n-1)$, $D_M$ is distributed as $D_{n,M}$. We use the following standard proposition to relate $D_{n,M}$ and $D(n,p)$ (see, for example,~\cite{rainbowarb,pittel}).

\begin{prop}\label{convert} Let $\mathcal{P}$ be a digraph property and let $p=M/n(n-1)$. If $M=M(n)\to\infty$ is any function such that $M(1-p)\to\infty$, then, for sufficiently large $n$,
\[
\P(D_{n,M}\text{ satisfies }\mathcal{P})\leq 5\sqrt{M}\cdot \P(D(n,p)\text{ satisfies }\mathcal{P}).
\]
\end{prop}

\subsection{Maximum and minimum degree}\label{sec71}
We will use the following standard result, which implies that, almost surely, each digraph $D_M$ in the $n$-vertex random digraph process with $M\leq n(\log n-\lg{2})$ is not Hamiltonian.

\begin{lemma}\label{when0}(See \cite{ER64mat}) If
$M=n(\log n-\lg{2})$, then, $D=D_{n,M}$ almost surely satisfies $\delta^+(D)=0$ or $\delta^-(D)=0$.
\end{lemma}

When $M\geq 50n\log n$, each digraph $D_M$ in the random digraph process will likely have well-bounded minimum and maximum degrees, as follows.

\begin{lemma}\label{min-degree}
In almost every $n$-vertex random digraph process $D_0,D_1,\ldots,D_{n(n-1)}$, if $M\geq 50n\log n$, then $\delta^\pm(D_M)\geq M/2n$ and $\Delta^\pm(D_M)\leq 2M/n$.
\end{lemma}
\begin{proof}
For each $M\geq 50n\log n$, let $p_M=M/n(n-1)$ and $\bar{D}_M=D(n,p_M)$. For each $v\in V(D)$ and $j\in \{+,-\}$, $\E(d^j_{\bar{D}_M}(v))=(n-1)p_M=M/n\geq 50\log n$, so that, by Lemma~\ref{chernoff},
\[
\P(d^j_{\bar{D}_M}(v)< M/2n\text{ or }d^j_{\bar{D}_M}(v)> 2M/n)\leq 2\exp(-50\log n/12)=o(n^{-4}).
\]
Therefore, by a union bound, with probability $1-o(n^{-3})$, $\delta^\pm(\bar{D}_M)\geq M/2n$ and $\Delta^\pm(\bar{D}_M)\leq 2M/n$.

Now, for $50n\log n\leq M\leq n(n-1)-\log n$, by Proposition~\ref{convert}, with probability $1-o(n^{-2})$, $\delta^\pm(D_M)\geq M/2n$ and $\Delta^\pm(D_M)\geq 2M/n$. Thus, by a union bound, this property almost surely holds for each $50n\log n\leq M\leq n(n-1)-\log n$ in the random digraph process. Finally, note that, for each $M\geq n(n-1)-\log n$, $\delta^\pm(D_{M})\geq n-1-\log n\geq M/2n$, for sufficiently large $n$, and $\Delta^\pm(D_{M})\leq n-1\leq 2M/n$.
\end{proof}

We also need a maximum in- and out-degree condition earlier in the random digraph process, as follows.

\begin{lemma}\label{maxdegree} For each $n\log n/2\leq M\leq 50n\log n$, if $D=D_{n,M}$, then $\P(\Delta^\pm(D)\leq 100M/n)=1-o(n^{-2})$.
Furthermore, then, in almost every $n$-vertex random digraph process $D_0,D_1,\ldots,D_{n(n-1)}$, every digraph $D_M$ with $M\geq n\log n/2$ has $\Delta^{\pm}(D)\leq 100M/n$.
\end{lemma}
\begin{proof} The required bounds on the maximum in- and out-degree almost surely hold for each $M\geq 50n\log n$ by Lemma~\ref{min-degree}.
For each $n\log n/2\leq M\leq 50n\log n$, let $p_M=M/n(n-1)$ and $\bar{D}_M=D(n,p_M)$. For each $v\in V(D)$ and $j\in \{+,-\}$, we have
\[
\P\left(d^j_{\bar{D}_M}(v)\geq \frac{100M}{n}\right)\leq \binom{n-1}{100M/n}p_M^{100M/n}\leq \left(\frac{enp_M}{100M/n}\right)^{100M/n}
\leq \left(\frac{e}{50}\right)^{100M/n}=o(n^{-4}).
\]
Therefore, with probability $1-o(n^{-3})$, $\Delta^\pm(\bar{D}_M)\leq 100M/n$.
Thus, by Proposition~\ref{convert}, for each $n\log n/2\leq M\leq 50n\log n$, with probability $1-o(n^{-2})$, $\Delta^\pm(D_{n,M})\leq 100M/n$, as required. Finally, by a union bound, this property almost surely holds for each $n\log n/2\leq M\leq 50n\log n$ in the random digraph process.
\end{proof}

\subsection{Edges between sets}\label{sec72}
We will use the following simple proposition on the typical number of edges between sets in $D(n,p)$.

\begin{prop}\label{msets} Let $\eps>0$, $p\geq 1/n$ and $m=\lg{4}/p$. Then, with probability $1-o(n^{-3})$ in $D(n,p)$, if sets $A,B\subset V(D_M)$ satisfy $|A|\geq m/2$ and $|B|\geq n/10$, then
\[
(1-\e)p|A||B|\leq e^\pm(A,B)\leq (1+\e)p|A||B|.
\]
\end{prop}
\begin{proof} For each such $A$ and $B$, and each $j\in \{0,1\}$, $e^j(A,B)$ is a binomial random variable with expectation $p|A||B|$. (Note that, even when $A$ and $B$ are not disjoint, each edge is counted exactly once in $e^j(A,B)$.) Thus, by Lemma~\ref{chernoff}, we have
\[
\P(|e^j(A,B)-p|A||B||>\e p|A||B|)\leq 2\exp(-\e^2p|A||B|/3)\leq 2\exp(-\e^2n\lg{4}/60)= 2\exp(-\omega(n)).
\]
There are at most $2(2^n)^2$ choices for $j\in \{0,1\}$ and such sets $A$ and $B$. Thus, by a union bound, the property in the proposition holds with probability $1-o(n^{-3})$.
\end{proof}


\subsection{Expansion from minimum degree conditions}\label{sec73}
We now prove a lemma used to show both \ref{pseud1} and \ref{pseud2} in Definition~\ref{pseuddefn}. The proof follows a section of the proof by Alon, Krivelevich and Sudakov of Lemma 3.1 in \cite{AKS07}.

\begin{lemma} \label{mindegexp}
Let
\[
p\geq \frac{\log n}{10n},\;\;\;\; d=\frac{p(n-1)}{2\cdot 10^3\log n}\geq 10^{-5},\;\;\;\;m=\frac{n\lg{3}}{d\log n}\;\;\;\; and\;\;\;\; f(n)=o(\lg{3}).
\]
Then, with probability $1-o(n^{-3})$, in $D=D(n,p)$, for any two disjoint sets $A,B\subset [n]$, with $|A|\leq 4m$, and any integer $k$ with $1\leq k\leq n/8m$, and any $j\in \{+,-\}$, if $e_D^j(A,B)\geq dk|A|\lg{2}/f(n)$, then $|B|\geq k|A|$.
\end{lemma}
\ifproofdone
\else
\pr
For each $k\in [n/8m]$, let $d_k=dk\lg{2}/f(n)$. If $D$ does not have the property in the lemma then there is some $k\leq n/8m$, $j\in \{+,-\}$ and two disjoint sets $A,B\subset V(D)$, where $|A|\leq 4m$, $|B|=k|A|$ and $e_{D}^j(A,B)\geq d_k|A|$ (adding vertices to $B$ if necessary to get equality). For each $r\in [4m]$, let $p_{r,k}$ be the probability no two such sets occur with $|A|=r\leq 4m$ (noting this does not depend on $j$). Then,
\begin{align}\allowdisplaybreaks
 p_{r,k} &\leq \binom{n}{r}\binom{n}{kr}\binom{kr^2}{d_kr}p^{d_kr}\nonumber
\\ &\leq\left(\frac{en}{r}\left(\frac{en}{kr}\right)^{k}\left(\frac{ekrp}{d_k}\right)^{d_k}\right)^r\nonumber
\\ &\leq\left(\left(\frac{en}{r}\right)^{2k}\left(\frac{ekrp}{d_k}\right)^{d_k}\right)^r
\nonumber
\\ &=\left(\left(\frac{e^2knp}{d_k}\right)^{2k}\left(\frac{ekrp}{d_k}\right)^{d_k-2k}\right)^r.\label{finalone}
\end{align}
Now,
\begin{equation}\label{add1}
\frac{e^2knp}{d_k}=\frac{e^2np\cdot f(n)}{d\lg{2}}=O\left(\frac{\log n\cdot f(n)}{\lg{2}}\right)=o(\log n).
\end{equation}
Furthermore,
\begin{equation}\label{add2}
\frac{ekrp}{d_k}=\frac{erp\cdot f(n)}{d\lg{2}}=O\left(\frac{r\log n\cdot f(n)}{n\lg{2}}\right)= o\left(\frac{r\cdot\lg{3}\cdot f(n)}{md\cdot\lg{2}}\right)=o\left(\frac{r}{md\cdot(\lg{2})^{1/2}}\right).
\end{equation}
For sufficiently large $n$, we have $d_k\geq 4k$. Therefore, by \eqref{finalone}, \eqref{add1} and \eqref{add2}, we have, for sufficiently large $n$,
\begin{align} p_{r,k}&\leq\left(\log^{2k}n\left(\frac{r}{md\cdot (\lg{2})^{1/2}}\right)^{d_k/2}\right)^r= \left(\log^{2}n\left(\frac{r}{md\cdot (\lg{2})^{1/2}}\right)^{d_k/2k}\right)^{kr}.\label{vienna}
\end{align}
If $r<\sqrt{n}$, then $r/md=O(\log n/\sqrt{n})$, and hence, as $d_k=\omega(k)$ and $kr\geq 1$, $p_{r,k}=o(n^{-4})$.

If $r\geq \sqrt{n}$, then, as $r\leq 4m$, we have, for large $n$, by \eqref{vienna}, that
\begin{align*} p_{r,k}&\leq\left(\log^{2}n\left(\frac{4}{d\cdot(\lg{2})^{1/2}}\right)^{d\lg{2}/2f(n)}\right)^{kr}
\leq\left(\log^{2}n\cdot\exp\left(-\frac{d\lg{2}\cdot\lg{3}}{8f(n)}\right)\right)^{kr}.
\end{align*}
As $d=\Omega(1)$ and $f(n)=o(\lg{3})$, we have that, for sufficiently large $n$,
$p_{r,k}\leq 2^{-kr}\leq 2^{-k\sqrt{n}}=o(n^{-4})$.

Therefore, $2\sum_{r,k}p_{r,k}=o(m\cdot (n/m)\cdot n^{-4})=o(n^{-3})$. Thus, the probability for some $j$, $r$ and $k$ that such a pair $A$, $B$ exists is $o(n^{-3})$.
\oof
\fi

\subsection{Low degree vertices}\label{sec74}

We will treat vertices with low in-degree or low out-degree separately. We will use that, typically, no vertex will have both low in-degree and low out-degree in the digraphs we consider.

\begin{prop}\label{summindeg}
In almost every $n$-vertex random digraph process $D_0,D_1,\ldots,D_{n(n-1)}$, if $M\geq 9n\log n/10$, then, for all $v\in V(D_M)$, $d^+_{D_M}(v)+d^-_{D_M}(v)\geq 2M/10^3n$.
\end{prop}
\begin{proof} Note that in almost every random digraph process this property holds for $M\geq 50n\log n$ by Lemma~\ref{min-degree}, so we need only show this almost surely holds for every $9n\log n/10\leq M\leq 50n\log n$.

Let $p=7\log n/8n$, $D=D(n,p)$ and $d=\log n/10$. Note that $20\leq e^3$. For each $v\in V(D)$, the probability that $d^+_{D}(v)+d^-_{D}(v)\leq d$ is at most
\begin{align*}
\sum_{i=0}^d\binom{2n-2}{i}p^i(1-p)^{2n-2-i}&\leq \sum_{i=0}^d\left(\frac{2enp}{i}\right)^ie^{-p(2n-2-i)}
\leq e^{-pn(2-o(1))}\cdot\sum_{i=0}^d\left(\frac{2e\log n}{i}\right)^i \\
&\leq (d+1)\cdot e^{-pn(2-o(1))} \cdot \left(\frac{2e\log n}{d}\right)^d\leq 2d\cdot e^{-(7/4-o(1))\log n}\cdot \left(20e\right)^{\log n/10}
\\
&\leq 2d\cdot e^{-(7/4-2/5-o(1))\log n}=o(n^{-1}).
\end{align*}
Thus, by a union bound, almost surely, for each $v\in V(D)$, $d^+_{D}(v)+d^-_{D}(v)\geq d$.

An easy application of Lemma~\ref{chernoff} demonstrates that $D$ almost surely has at most $9n\log n/10$ edges. Furthermore, the property -- $\mathcal{P}$ say -- that, for each $v\in V(D)$, $d^+_{D}(v)+d^-_{D}(v)\geq d$, is an increasing property.
Thus, we have, with $M_0=9n\log n/10$,
\begin{align*}
1-o(1)=\P(D(n,p)\in \mathcal{P})\leq \P(e(D(n,p))> M_0)+\P(D_{n,M_0}\in \mathcal{P})=o(1)+\P(D_{n,M_0}\in \mathcal{P}).
\end{align*}
Hence, we have $\P(D_{n,M_0}\in \mathcal{P})=1-o(1)$.
Therefore, almost surely, if $M_0=9n\log n/10\leq M\leq 50n\log n$, then, for every $v\in V(D_M)$, $d^+_{D_M}(v)+d^-_{D_M}(v)\geq d=\log n/10\geq 2M/10^3n$.
\end{proof}

We will collect the vertices of low in- or out-degree in the random digraph into a set $S$. We use the following definition to record that there will typically be no vertices in $S$ which are close together (in a graph-theoretic sense).

\begin{defn} For a vertex set $S$ in a digraph $D$, an \emph{$S$-path} is a path with length at most 4 in $D$ (with any orientation on the edges) starting and ending in $S$. An \emph{$S$-cycle} is a cycle with length at most 4 in $D$ (with any orientations on the edges, and a cycle with length 2 permitted if it consists of two distinct edges) which contains a vertex in $S$.
\end{defn}

We wish to show that, in almost every $n$-vertex random digraph process, each digraph $D_M$ with $9n\log n/10\leq M\leq 50\log n$ has no $S$-paths or $S$-cycles, when $S$ is the set of vertices with low in- or out-degree, and $S$ is a small set (see Lemma~\ref{finalS}). To do this, we cannot show this is likely for each such digraph and then take a union bound, as the property is not sufficiently likely. Instead, we start by showing that this property is likely in a certain random digraph $D(n,p)$.

\begin{lemma}\label{noApaths} If $p=7\log n/8n$, then, almost surely, the following holds for $D=D(n,p)$ with $S=\{v\in V(D):d^+(v)< \log n/20\text{ or }d^-(v)< \log n/20\}$. There are no $S$-paths or $S$-cycles and $|S|\leq n^{1/3}$.
\end{lemma}
\begin{proof}
Let $d=\log n/20$. Note that $20\leq e^3$.
We have
\begin{align*}
\E|S|&\leq n \cdot \sum_{i=0}^{d}2\binom{n-1}{i}p^i(1-p)^{n-1-i}\leq 2n(d+1)\cdot\left(\frac{enp}{d}\right)^{d}\cdot e^{-(1-o(1))np}\\
&\leq n\log n\cdot\left(20e\right)^{d}\cdot e^{-(7/8-o(1))\log n}
\leq \log n\cdot e^{4d}\cdot e^{(1/8+o(1))\log n}\\
&=\log n\cdot e^{(1/5+1/8+o(1))\log n}=o(n^{1/3}).
\end{align*}
Thus, by Markov's inequality, we almost surely have that $|S|\leq n^{1/3}$.

Let $X$ be the number of $S$-paths in $D(n,p)$. Then,
\begin{align*}
\E X&\leq \binom{n}{2}\sum_{k=0}^3(2p)^{k+1}n^{k}\cdot\sum_{i=0}^{2d}4\binom{2n}{i}p^i(1-p)^{2n-5-i}
\leq n(2np)^4\cdot(2d+1)\cdot 4\left(\frac{2enp}{2d}\right)^{2d}e^{-(2-o(1))np}\\
&=O\left( n\log^5n \cdot (20e)^{2d}e^{-(2-o(1))np}\right)
=O(n\log^5n\cdot \exp(8d-(7/4-o(1))\log n))\\
&=O(n\log^5n\cdot \exp(5\log n/4))=o(1).
\end{align*}
Thus, almost surely, there are no $S$-paths in $D$.

Let $Y$ be the number of $S$-cycles in $D(n,p)$. Then, similarly,
\begin{align*}
\E Y&\leq n\cdot\sum_{k=1}^3(2p)^{k+1}n^k\cdot\sum_{i=0}^{d}2\binom{n}{i}p^i(1-p)^{n-4-i}
=O\left((np)^4\cdot(d+1)\cdot\left(\frac{enp}{d}\right)^{d}\cdot e^{-(1-o(1))np}\right)
\\
&=O\left( \log^5n\cdot\left(20e\right)^{d}\cdot e^{-(7/8-o(1))\log n}\right)
=O\left( \log^5n\cdot \exp(4d-(7/8-o(1))\log n)\right)=o(1).
\end{align*}
Thus, almost surely, there are no $S$-cycles in $D$.
\end{proof}

Lemma~\ref{noApaths} shows only that a property exists with probability $1-o(1)$, so we cannot use Proposition~\ref{convert} as we did before. However, we can easily show that the property holds for some useful digraph in the random digraph process, as follows.

\begin{corollary}\label{M0exists} There exists some $M_0=M_0(n)$ with $n\log n/2\leq M_0\leq 9n\log n/10$ such that the following almost surely holds for $D=D_{n,M_0}$ with $S=\{v\in V(D):d^+(v) < \log n/20\text{ or }d^-(v)> \log n/20\}$. There are no $S$-paths or $S$-cycles and $|S|\leq n^{1/3}$.
\end{corollary}
\begin{proof} Let $p=7\log n/8n$, $N_0=n\log n/2$ and $N_1=9n\log n/10$. By a simple application of Lemma~\ref{chernoff}, we have that, almost surely, $N_0\leq e(D(n,p))\leq N_1$. Thus, if $\mathcal{P}$ is the property of digraphs satisfying the condition in the corollary, then
\begin{align*}
\P(D(n,p)\in \mathcal{P})&=\sum_{M=0}^{n(n-1)}\P(D_{n,M}\in \mathcal{P})\cdot \P(e(D(n,p)=M)\\
&\leq \P(e(D(n,p))\notin (N_0,\ldots,N_1))+\sum_{M=N_0}^{N_1}\P(D_{n,M}\in \mathcal{P})\cdot \P(e(D(n,p))=M)\\
&\leq o(1)+\left(\sup_{N_0\leq M\leq N_1}\P(D_{n,M}\in \mathcal{P})\right)\cdot\sum_{M=N_0}^{N_1} \P(e(D(n,p))=M)\\
&\leq o(1)+\sup_{N_0\leq M\leq N_1}\P(D_{n,M}\in \mathcal{P}).
\end{align*}
Thus, as $\P(D(n,p)\in \mathcal{P})=1-o(1)$, we must have $\sup_{N_0\leq M\leq N_1}\P(D_{n,M}\in \mathcal{P})=1-o(1)$. Choosing $M_0$ to maximise $\P(D_{n,M_0}\in \mathcal{P})$ subject to $N_0\leq M_0\leq N_1$ thus gives the result.
\end{proof}

Next, by starting with the digraph $D_{n,M_0}$ from Corollary~\ref{M0exists}, we show that if $O(n\log n)$ random edges are added then it is likely that no short paths between the vertices with small in- or out-degree are created. This will give us the following lemma.

\begin{lemma}\label{finalS} Almost surely, in the $n$-vertex random digraph process $D_0,D_1,\ldots,D_{n(n-1)}$, the following holds for each $M$ with $9n\log n/10\leq M\leq 50n\log n$. Letting $S_M=\{v\in V(D):d^+_{D_M}(v)< \log n/20\text{ or }d^-_{D_M}(v)<\log n/20\}$, there are no $S_M$-paths or $S_M$-cycles, and $|S_M|\leq n^{1/3}$.
\end{lemma}
\begin{proof}
Let $M_0$ be from Corollary~\ref{M0exists}. Let $N_0=9n\log n/10$ and $N_1=50n\log n$, and note that $M_0\leq N_0$. Reveal the edges of $D_{M_0}$ and let $S=S_{M_0}$,
 so that, almost surely, $|S|\leq n^{1/3}$ and there are no $S$-paths or $S$-cycles. Note that $S_M\subset S$ for each $M_0\leq M\leq N_1$. Thus, if we can show that, almost surely, $D_{N_1}$ has no $S$-paths or $S$-cycles then we are done.

For each $M$, $M_0\leq M< N_1$, let $e_M$ be the edge added to $D_M$ to get $D_{M+1}$. Let $S'_M$ be the set of vertices within a graph distance 2 of any $u\in S_M$ in the underlying undirected graph of $D_M$. Let $E_M$ be the event that $e_M$ is contained within $S'_M$. Note that if no such event $E_M$, $M_0\leq M<N_1$, occurs, then $D_{N_1}$ has no $S$-paths or $S$-cycles.

Note that, for each $M$, $M_0\leq M< N_1$, $|S'_M|\leq 2(\Delta^+(D_M)+\Delta^-(D_M))^2|S|$, and thus
\[
\P(E_M|(\Delta^\pm(D_M)\leq 10^4\log n)\land( |S|\leq n^{1/3}))
=O\left(\frac{\log^4n\cdot(n^{1/3})^2}{n(n-1)-M}\right)=O\left(\frac{n^{2/3}\log^4n}{n^2}\right)=o\left(\frac{1}{n\log n}\right).
\]
By Lemma~\ref{maxdegree}, for each $M_0\leq M< N_1$, $\P(\Delta^\pm(D_M)\leq 10^4\log n)=1-o(n^{-2})$.
Thus, the probability that no event $E_M$, $M_0\leq M< N_1$, occurs is at most
\[
\P(|S|>n^{1/3})+\sum_{M=M_0}^{N_1-1}\left(o(n^{-2})+\P(E_M|(\Delta^\pm(D_M)\leq 10^4\log n)\land (|S|\leq n^{1/3})\right)=o(1).\qedhere
\]
\end{proof}

\subsection{Resilience in the random digraph process}\label{sec75}

For convenience, we collect together the properties of the random digraph process that we have shown into the following corollary.
\begin{corollary}\label{collect} Let $\eps>0$. In almost every $n$-vertex random digraph process $D_0,D_1,\ldots,D_{n(n-1)}$, each digraph $D_M$ with $\delta^\pm(D_M)\geq 1$ satisfies $M\geq n(\log n-\lg{2})$ and the following with
\[
d_M=\frac{M}{2\cdot 10^3n\log n},\;\;\;\; m_M=\frac{n\lg{3}}{d_M\log n},\;\;\;\; and \;\;\;\; S_M=\{v:d^+_{D_M}(v)< 2d_M\log n\text{ or }d^-_{D_M}(v)< 2d_M\log n\}.
\]
\stepcounter{alabel}
\begin{enumerate}[label = \textbf{\Alph{alabel}\arabic{enumi}}]
\item\label{pseudr4} $\Delta^\pm(D_M)\leq 100 M/n\leq 10^6d_M\log n/2$.
\item\label{pseudr5} For each $v\in V(D_M)$, $d^+_{D_M}(v)+d^-_{D_M}(v)\geq 2M/10^3n=4d_M\log n$.
\item\label{pseudr0} There are no $S_M$-paths or $S_M$-cycles in $D_M$ and $|S_M|\leq \sqrt{n}$.
\item\label{pseudr3} If sets $A,B\subset V(D_M)$ satisfy $|A|\geq m_M/2$ and $|B|\geq n/4$, then, for each $j\in \{+,-\}$, letting $p_M=M/n(n-1)$, $(1-\e/100)p_M|A||B|\leq e^j_{D_M}(A,B)\leq (1+\e/100)p_M|A||B|$.
\item\label{pseudr1} For any disjoint sets $A,B\subset V(D_M)$ and $j\in \{+,-\}$ with $|A|\leq 4m_M$ and, for each $v\in A$, $d^j_{D_M}(v,B)\geq d_M\lg{2}/4\lg{4}$, we have $|B|\geq 10|A|$.
\item\label{pseudr2} For any disjoint sets $A,B\subset V(D_M)$ with $|A|\leq 4m_M$ and $j\in \{+,-\}$, for each $v\in A$, $d^j_{D_M}(v,B)\geq d_M(\log n)^{2/3}/4$, we have $|B|\geq |A|(\log n)^{1/3}$.
\end{enumerate}
\end{corollary}
\begin{proof} By Lemma~\ref{maxdegree}, we almost surely have that $\Delta^\pm(D_M)\leq 100M/n$, and thus \ref{pseudr4} holds, for each $M\geq n\log n/2$. By Lemma~\ref{summindeg}, we almost surely have that \ref{pseudr5} holds for all $M\geq 9n\log n/10$.
Almost surely, by Lemma~\ref{min-degree}, for each $M\geq 50n\log n$, $\delta^{\pm}(D_M)\geq M/2n\geq 10^3d_M\log n$, so that $S_M=\emptyset$. Note that, if $M\leq 50n\log n$, then $2d_M\log n\leq \log n/20$, and, thus combining this with Lemma~\ref{finalS}, we have that, almost surely, \ref{pseudr0} holds for each $M\geq 9n\log n/10$.
By Proposition~\ref{msets} and Proposition~\ref{convert}, we have that \ref{pseudr3} almost surely holds for each $9n\log n/10\leq M\leq n(n-1)-\log n$. Note that, when $M\geq n(n-1)-\log n$, there are at most $\log n$ missing edges, and so \ref{pseudr3} easily holds for sufficiently large $n$.

Almost surely, by Lemma~\ref{mindegexp} with $f(n)=40\lg{4}$ and taking the resulting property with $k=10$, and using Proposition~\ref{convert}, \ref{pseudr1} holds for each $9n\log n/10\leq M\leq n(n-1)-\log n$. Almost surely, by Lemma~\ref{mindegexp} with $f(n)=4\lg{2}/(\log n)^{1/3}$
and taking the resulting property with $k=(\log n)^{1/3}=o(n/m_M)$, and using Proposition~\ref{convert}, \ref{pseudr2} holds for each $9n\log n/10\leq M\leq n(n-1)-\log n$.
Note that, if $M\geq n(n-1) -\log n$, then $d_M\lg{2}/4\lg{4}\geq 40m_M$ and $d_M(\log n)^{2/3}/4\geq 4m_M(\log n)^{1/3}$, so that \ref{pseudr1} and \ref{pseudr2} hold.

Therefore, \ref{pseudr4}--\ref{pseudr2} almost surely hold for each $M\geq 9n\log n/10$. By Lemma~\ref{when0}, almost surely, if  $M\leq  n(\log n-\lg{2})$, then $\delta^+(D_M)=0$ or $\delta^-(D_M)=0$. Thus, almost surely, if $\delta^\pm(D_M)\geq 1$,
then $M\geq n(\log n-\lg{2})$ and  \ref{pseudr4}--\ref{pseudr2} hold.
\end{proof}

We now have the tools we need to prove the resilience in Theorem~\ref{directedresilience}.

\begin{proof}[Proof of the resilience in Theorem~\ref{directedresilience}]
Almost surely, by Corollary~\ref{collect}, every digraph $D_M$ in the random digraph process with $\delta^\pm(D_M)\geq 1$ satisfies $M\geq n(\log n-\lg{2})$  and \ref{pseudr4}--\ref{pseudr2} with
\[
d_M=\frac{M}{2\cdot 10^3n\log n},\;\;\;\; m_M=\frac{n\lg{3}}{d_M\log n}\;\;\;\; and \;\;\;\; S_M=\{v:d^+_{D_M}(v)< 2d_M\log n\text{ or }d^-_{D_M}(v)< 2d_M\log n\}.
\]
We will show, for sufficiently large $n$, that each such $D_M$ is $(1/2-\eps)$-resiliently Hamiltonian.

Fix $M\geq n(\log n-\lg{2})$ such that $\delta^\pm(D_M)\geq 1$, then, and suppose that $H\subset D_M$ with $d^j_H(v)\leq (1/2-\eps)d^j_{D_M}(v)$ for each $v\in V(D_M)$ and $j\in \{+,-\}$.
Pick for each $v\in S_M$ an in-neighbour $x_v$ and an out-neighbour $y_v$ in $D_M-H$, noting that this is possible as $d^j_{D_M-H}(v)>d^j_{D_M}(v)/2>0$ for each $j\in \{+,-\}$. Note that, as there are no $S_M$-paths or $S_M$-cycles in $D_M$ by \ref{pseudr0}, the vertices $v$, $x_v$ and $y_v$, $v\in S_M$, are distinct.
Form $D$ from $D_M-H$ by deleting the vertices in $S_M$ and, for each $v\in S_M$, merging $x_v$ into $y_v$ to get the new vertex $z_v$.

\begin{claim}\label{crux}
For sufficiently large $n$, $D$ is an $(\eps/100,d_M)$-pseudorandom digraph.
\end{claim}
\begin{proof} We will check the conditions \ref{pseud0}--\ref{pseud3}. Let $\bar{n}=|D|=n-2|S_M|=(1-o(1))n$, where we have used \ref{pseudr0}. Let $m=\bar{n}\log^{[3]}\bar{n}/d_M\log \bar{n}$, so that, for sufficiently large $n$,
 $m_M\geq m\geq m_M/2$. Note that, as there are no $S_M$-paths or $S_M$-cycles in $D_M$ by \ref{pseudr0}, every vertex in $V(D_M)\setminus S_M$ has at most 1 in- or out-neighbour in $S_M\cup \{x_v,y_v:v\in S_M\}$ in $D_M$. Thus, every vertex in $D$ has in- and out- degree at least $(1/2+\eps)\cdot 2d_M\log n-1\geq d_M\log n\geq d_M\log \bar{n}$ in $D$, for sufficiently large $n$.
Furthermore, by \ref{pseudr4}, $\Delta^{\pm}(D)\leq 10^6d_M\log \bar{n}$, and therefore \ref{pseud0} holds.

Now, suppose $A,B\subset V(D)$ are disjoint sets with $|A|\leq 2m\leq 4m_M$ and, for each $v\in A$, $d_D^+(v,B)\geq d_M\log^{[2]}\bar{n}/\log^{[4]}\bar{n}$. Let $A'\subset V(D_M)$ be formed from $A$ by replacing any vertex $z_v$, $v\in S_M$, with the vertex $y_v$, and let $B'\subset V(D_M)$ be formed from $B$ by replacing any vertex $z_v$, $v\in S_M$, with the vertex $x_v$. Then, for each $v\in A'$, we have $d^+_{D_M}(v,B')\geq d_M\lg{2}/4\lg{4}$, for sufficiently large $n$. Thus, using \ref{pseudr1}, we have $|B|=|B'|\geq 10|A'|=10|A|$.
Similar reasoning for $d^-_{D_M}(v,B')$, again using \ref{pseudr1}, completes the proof to show that \ref{pseud1} holds. Similarly, \ref{pseud2} follows from \ref{pseudr2}.

Therefore, it is left only to show that \ref{pseud3} holds. Suppose, for contradiction, there is some $U\subset V(D)$, with $|U|=m$, for which, without loss of generality $|N^+_{D}(U)|\leq (1/2+\eps/100)\bar{n}$. Let $U'$ be the set $U$ with any vertex $z_v$, $v\in S_M$, replaced by $y_v$, so that $|U'|=m$ and, using \ref{pseudr0},
\[
|N^+_{D_M-H}(U')|\leq |N^+_{D}(U)|+2|S_M|\leq (1/2+\eps/100)\bar{n}+2\sqrt{n}.
\]
Let $V=V(D_M)\setminus (U'\cup N^+_{D_M-H}(U'))$, so that
\begin{equation}\label{last}
|V|\geq n-m-(1/2+\eps/100)n-2\sqrt{n}\geq (1/2-\eps/50)n,
\end{equation}
for sufficiently large $n$. By \ref{pseudr3}, letting $p_M=M/n(n-1)$, we have
\begin{equation}\label{mission}
e^+_{D_M}(U',V)\geq (1-\eps/100 )p_M|U'||V|\overset{\eqref{last}}{\geq} (1-\eps/100)\cdot(1/2-\eps/50 )\cdot p_M|U'|n.
\end{equation}
On the other hand, by the choice of $V$, we have $e^+_{D_M-H}(U',V)=0$, so that
\begin{align}
e^+_{D_M}(U',V)&= e^+_{H}(U',V)= \sum_{u\in U'}d^+_H(u,V)\leq \sum_{u\in U'}(1/2-\eps)d^+_{D_M}(u)
\nonumber\\
&=(1/2-\eps)e^+_{D_M}(U',V(D))
\leq(1/2-\eps)\cdot (1+\eps/100)\cdot p_M|U'| n,\label{mission2}
\end{align}
where the last inequality follows by \ref{pseudr3}.

Thus, we have, by \eqref{mission} and \eqref{mission2},
\[
(1-\eps/100)\cdot (1/2-\eps/50 )\leq(1/2-\eps)\cdot (1+\eps/100),
\]
a contradiction. Therefore, \ref{pseud3} holds, completing the proof that $D$ is $(\eps/100,d)$-pseudorandom.
\end{proof}

Thus, by Theorem~\ref{pseudorandom}, for sufficiently large $n$, $D$ contains a directed Hamilton cycle, $C$ say. For each vertex $z_v$, $v\in S_M$, in $C$, replace $z_v$ by $x_vvy_v$. This gives a directed Hamilton cycle in $D_M-H$. Thus, $D_M-H$ is Hamiltonian, and therefore $D_M$ is $(1/2-\eps)$-resiliently Hamiltonian.
\end{proof}

\subsection{Non-resilience in the random digraph process}\label{secnonres}
To prove the non-resilience of our random digraph, we divide the vertices without low in- or out-degree into two sets independently at random with equal probability. By Lemma~\ref{divisionlemma}, there must be such a partition where each vertex has roughly an equal number of in- and out-neighbours in each set. Carefully dividing the low degree vertices between these sets, we reach a bipartition where deleting the edges across this partition does not remove substantially more than one half of the in- and out-neighbours around any one vertex. This gives us a limit for the resilience of Hamiltonicity in the digraph.

\begin{proof}[Proof of the non-resilience in Theorem~\ref{directedresilience}]
Almost surely, by Corollary~\ref{collect}, every digraph $D_M$ in the random digraph process with $\delta^\pm(D_M)\geq 1$ satisfies $M\geq n(\log n-\lg{2})$ and \ref{pseudr4}--\ref{pseudr2} with
\[
d_M=\frac{M}{2\cdot 10^3n\log n},\;\;\;\; m_M=\frac{n\lg{3}}{d_M\log n}\;\;\;\; and \;\;\;\; S_M=\{v:d^+_{D_M}(v)< 2d_M\log n\text{ or }d^-_{D_M}(v)< 2d_M\log n\}.
\]

Take the vertices in $V(D_M)\setminus S_M$ and partition them as $A\cup B$ so that each vertex is placed into $A$ or $B$ independently at random with probability $1/2$. We will show that the following claim holds for sufficiently large $n$.

\begin{claim}\label{divideoclaim}
With positive probability, the following holds.
\stepcounter{alabel}
\begin{enumerate}[label = \textbf{\Alph{alabel}}]
\item For each $v\in V(D_M)$ and $j\in \{+,-\}$, if $d^j_{D_M}(v)\geq 2d_M\log n$, then $d^j_{D_M}(v,A),d^j_{D_M}(v,B)\geq (1/2-\eps)d^j_{D_M}(v)$.\label{divideor}
\end{enumerate}
\end{claim}
\begin{proof}[Proof of Claim~\ref{divideoclaim}]
For each $v\in V(D_M)$ and $j\in \{+,-\}$, let $E_{v,j}$ be the event that, if $d^j_{D_M}(v)\geq 2d_M\log n$, then either $d^j_{D_M}(v,A)<(1/2-\eps)d^j_{D_M}(v)$ or $d^j_{D_M}(v,B)< (1/2-\eps)d^j_{D_M}(v)$. Let $G$ be the dependence graph of these events and note that, by \ref{pseudr4}, $\Delta(G)\leq (\Delta^+(D_M)+\Delta^-(D_M))^2\leq 10^5M^2/n^2$.

Let $q=4\exp(-\eps^2d_M\log n/100)$, so that $q\cdot \Delta(G)=o(1)$.
By \ref{pseudr0}, $d^j_{D_M-S_M}(v)\geq d^j_{D_M}(v)-1\geq 2d_M\log n-1\geq d_M\log n$.
By Lemma~\ref{chernoff}, we then have
\[
\P(E_{v,j})\leq 2\exp\left(-\frac{\eps^2}{12}\cdot \frac{d^j_{D_M-S_M}(v)}{2}\right)\leq 2\exp\left(-\frac{\eps^2d_M\log n}{100}\right)=q/2\leq q(1-q\cdot\Delta(G))\leq q(1-q)^{\Delta(G)}.
\]
Thus, by Theorem~\ref{genlocal}, no event $E_{v,j}$ occurs with positive probability.
\end{proof}

Thus, by Claim~\ref{divideoclaim}, there is some partition $V(D_M)\setminus S_M=A\cup B$ such that \ref{divideor} holds. Note that $A,B\neq\emptyset$.  For each $v\in S_M$, by \ref{pseudr5}, there is $j_v$ such that $d^{j_v}_{D_M}(v)\geq 2d_M\log n$, and thus, by \ref{divideor}, $d^{j_v}_{D_M}(v,A),d^{j_v}_{D_M}(v,B)\geq (1/2-\eps)d^{j_v}_{D_M}(v)$. Let $i_v\in \{+,-\}$ such that $i_v\neq j_v$. By \ref{pseudr0}, as $v\in S_M$, there are no edges from $v$ to $S_M$ in either direction. Thus, there is some $X_v\in \{A,B\}$ such that $d^{i_v}_{D_M}(v,X_v)\geq d^{i_v}_{D_M}(v)/2$.

Let $A'=A\cup \{v\in S_M:X_v=A\}$ and $B'=B\cup \{v\in S_M:X_v=B\}$, and note that, for each $X\in \{A',B'\}$, $v\in X$ and $j\in \{+,-\}$,  $d^j_{D_M}(v,X)\geq (1/2-\eps)d^j_{D_M}(v)$.
Let $H$ be the bipartite digraph with vertex classes $A'$ and $B'$, with edges exactly those edges in $D_M$ between $A'$ and $B'$ in either direction. Thus, we have,
for each $v\in V(D_M)$ and $j\in \{+,-\}$, $d^j_H(v)\leq (1/2+\eps)d^j_{D_M}(v)$.

As $A,B\neq\emptyset$, $D_M-H$ is disconnected and hence not Hamiltonian. Therefore, $D_M$ is not $(1/2+\eps)$-resiliently Hamiltonian.
\end{proof}

\bibliographystyle{plain}
\bibliography{dirbib}

\end{document}